\documentclass[reqno,12pt]{amsart}
\usepackage{amsmath,amssymb,epsfig,color}

\numberwithin{equation}{section}

\newtheorem{theorem}{Theorem}[section]

\newtheorem{lemma}[theorem]{Lemma}
\newtheorem{proposition}[theorem]{Proposition}
\newtheorem{definition}[theorem]{Definition}
\newtheorem{corollary}[theorem]{Corollary}

\theoremstyle{remark}
\newtheorem{example}[theorem]{Example}
\newtheorem{remark}[theorem]{Remark}

\newcommand{\Bip}{\operatorname{Bip}}
\newcommand{\eps}{\varepsilon}

\newcommand{\inc}[1]{\sim _{#1}}

\newcommand{\rank}{\operatorname{rk}}
\newcommand{\rest}[2] {#1\left| {_{#2}}\right. }
\newcommand{\vanish}[1]{}

\begin{document}

\title{The poset of bipartitions}

\author[G\'abor Hetyei and Christian Krattenthaler]{G\'abor Hetyei
\and
Christian Krattenthaler$^\dagger$
}

\address{Department of Mathematics and Statistics,
  UNC-Charlotte, Charlotte NC 28223-0001.
WWW: \tt http://www.math.uncc.edu/\~{}ghetyei/.}

\address{Fakult\"at f\"ur Mathematik, Universit\"at Wien,
Nordbergstra{\ss}e~15, A-1090 Vienna, Austria.
WWW: \tt http://www.mat.univie.ac.at/\~{}kratt.}

\thanks{$^\dagger$Research partially supported by the Austrian
Science Foundation FWF, grants Z130-N13 and S9607-N13,
the latter in the framework of the National Research Network
``Analytic Combinatorics and Probabilistic Number Theory"}



\begin{abstract}
Bipartitional relations were introduced by
Foata and Zeilberger in their characterization
of relations which give rise to equidistribution of the associated
inversion statistic and 
major index. We consider the natural partial order on
bipartitional relations given by inclusion. We show that, with respect
to 
this partial order, the bipartitional relations on a set of size $n$
form a graded lattice of rank $3n-2$. Moreover, we prove that the order
complex of this lattice is homotopy equivalent to a sphere of
dimension $n-2$. 
Each proper interval in this lattice has either a contractible order
complex, or it is isomorphic to the direct product of Boolean lattices
and smaller lattices of bipartitional relations.
As a consequence, we obtain that
the M\"obius function of every interval is $0$, $1$, or $-1$.
The main tool in the proofs is discrete Morse theory as developed by
Forman, and an application of this theory to order complexes of graded
posets, designed by Babson and Hersh, in the extended form of
Hersh and Welker.
\end{abstract}

\maketitle

\section{Introduction}

The poset of {\it partitions} $\Pi_n$ of the set $\{1,2,\dots,n\}$, 
where the order is
defined by refinement, is a classical object in combinatorics.
Various aspects of this poset have been studied in the literature
(cf.\ \cite[Ch.~3]{Stanley-EC1}). In particular, its M\"obius function
has been computed by Sch\"utzenberger and by Frucht and Rota
independently (cf.\ \cite[p.~359]{RotaAE}), and the homotopy type of its
order complex is a wedge of spheres. (The latter follows from the
well-known fact that $\Pi_n$ is a geometric lattice, and from
Bj\"orner's result \cite{BjoeAA} that geometric lattices are
shellable.) 

Closely related, and more relevant to the present work,  
is the poset of {\it ordered partitions} of $\{1,2,\dots,n\}$. 
It has a much simpler structure; for example, all intervals in this
poset are isomorphic to products of Boolean lattices.

{\it Bipartitional relations} ({\it bipartitions}, for short)  
were introduced by Foata and
Zeilberger~\cite{FoZeAE}, 
who showed that these are the relations $U$ for
which the (appropriately generalized) major index
$\operatorname{maj}_U$ and 
inversion number $\operatorname{inv}_U$ are equidistributed
on all rearrangement classes. Han \cite[Th.~5]{GNH} showed that these
bipartitional relations can be axiomatically characterized 
as the relations $U$ for which $U$ {\it and\/} its complement are
transitive.
(Cf.\ \cite{Foata--Kratt,Kasraoui} for further work on questions of
this kind.)

\begin{figure}[h]
\begin{center}
\input{bip2.pstex_t}
\end{center}
\caption{$\Bip(\{1,2\})$}
\label{fig:bip2}
\end{figure}

Bipartitional relations on $\{1,2,\dots,n\}$ 
carry a natural poset structure, the 
partial order being defined by inclusion of relations.
We denote the corresponding {\it poset of bipartitions} 
by $\Bip(\{1,2,\dots,n\})$. Figure~\ref{fig:bip2} shows the Hasse diagram of 
$\Bip(\{1,2\})$. The poset $\Bip(\{1,2,\dots,n\})$
contains the poset of ordered partitions of $\{1,2,\dots,n\}$
{\it and\/} its dual as subposets, and therefore can be considered as
a common extension of the two. It turns out that the richness of the
structure of the poset of bipartitions is comparable to that of the
lattice of partitions. To begin with, $\Bip(\{1,2,\dots,n\})$ is a
graded lattice of total rank $3n-2$
(see Theorem~\ref{thm:lattice} and Corollary~\ref{cor:rank}), 
although it is neither modular (cf.\
Example~\ref{ex:modular}) nor geometric, in fact not even 
Cohen--Macaulay (cf.\ Corollary~\ref{cor:nshell}). 
Furthermore, the
M\"obius function of each interval is $0$, $1$, or $-1$
(see Definition~\ref{def:regular}, Corollaries~\ref{C_MobBip} and 
\ref{C_Mobprod}, and Theorem~\ref{T_irreg} for the precise statement of
which intervals take which M\"obius function values). We show this
by proving the stronger result that the order complex of 
$\Bip(\{1,2,\dots,n\})$ is homotopy equivalent to a sphere
(see Theorem~\ref{T_fullc}), and each
proper interval is either the direct product of Boolean lattices and
smaller lattices of bipartitions, or has a contractible order
complex (see Proposition~\ref{P_regprod} and Theorem~\ref{T_irreg}).
The proofs of these facts form the most difficult part of our paper. 
They are essentially based on an adaptation of the Gray code of permutations
due to Johnson \cite{JohnAZ} and Trotter \cite{Trotter} and on 
work of Babson and Hersh \cite{Babson-Hersh}
(in the extended form by Hersh and Welker
\cite{Hersh-Welker}) constructing a discrete
Morse function in the sense of Forman 
\cite{Forman-advmath,Forman-trans-ams,Forman-slc} for the order complex of a
graded poset. The former is
needed to decompose $\Bip(\{1,2,\dots,n\})$ into a union of
distributive lattices in a shelling-like manner. This decomposition is
then refined using the well-known $EL$-shelling of distributive lattices
in order to obtain an enumeration of the maximal chains of
$\Bip(\{1,2,\dots,n\})$ to which the
results of Babson, Hersh, and Welker apply. 
(As Example~\ref{e_nonplo} shows, our enumeration of the maximal chains
of $\Bip(\{1,2,\dots,n\})$ is not a poset lexicographic order in the
sense of \cite{Babson-Hersh}, so that we do indeed need the extended
form observed in \cite{Hersh-Welker}.)
We remark that our ``two-phase" decomposition is
similar in spirit as constructions by Hanlon, Hersh and Shareshian
\cite{HaHSAA} and by Hersh and Welker \cite{Hersh-Welker}.
It would be interesting to see whether there is a uniform framework
for this type of shelling-like decompositions. However, we have not
been able to find such a generalization.

This paper is organized as follows. The next two sections are of
preliminary nature. Namely, Section~\ref{sec:def} reviews basic
facts on bipartitional relations, while Section~\ref{sec:Morse}
outlines the basic ideas of the construction of Babson and Hersh.
Here we observe that the proofs of their main results are actually
applicable to a larger class of enumerations of maximal chains, which
we call ``enumerations growing by creating skipped intervals."
In Section~\ref{sec:bipart}, we provide the proof that
$\Bip(\{1,2,\dots,n\})$ is a lattice, and we show that it is graded
and compute its rank function in Section~\ref{sec:rank}. 
The purpose of Section~\ref{sec:pi-compatible} is to show that 
$\Bip(\{1,2,\dots,n\})$ may be written as union of $n!$ distributive
lattices, each indexed by a permutation, where the proof of 
distributivity is deferred to Section~\ref{sec:distr}. We begin
Section~\ref{sec:JT} by reviewing the Johnson--Trotter algorithm and
an easy generalization to enumerating all elements in a direct product
of symmetric groups. We continue by using these enumerations to
decompose the order complex of $\Bip(\{1,2,\dots,n\})$, and the order
complex of certain intervals in it, in a shelling-like manner. 
Section~\ref{s_toc} forms the core of our article. Here we construct
an enumeration of the maximal chains of $\Bip(\{1,2,\dots,n\})$ that
refines the ``J--T decomposition" introduced in Section~\ref{sec:JT}, and
to which the results of Babson and Hersh are adaptable, as reviewed in
Section~\ref{sec:Morse}. Finally, in Section~\ref{sec:int}, we outline
how the argument of the preceding section may be modified to handle the
case of proper intervals of $\Bip(\{1,2,\dots,n\})$ as well. 

\section{Definition and elementary properties of bipartitional
relations} 
\label{sec:def}

In Definition~\ref{D_bip} below, we 
formally introduce bipartitional relations.
This definition is (essentially) taken from 
Han \cite{GNH}. Subsequently, we
shall provide a different way to see bipartitional relations, namely
in terms of ordered bipartitions. Historically, bipartitional
relations were originally defined by Foata and Zeilberger in
\cite[Def.~1]{FoZeAE} in the latter way, and Han showed in
\cite[Th.~5]{GNH} the equivalence with a condition which, in its turn,
is equivalent to the transitivity condition
that we use for defining bipartitional relations as given below.

\begin{definition}
\label{D_bip}
A relation $U\subseteq X\times X$ on a finite set $X$ is a {\em
bipartitional relation}, if  both $U$ and $(X\times X)\setminus U$ are
transitive. We denote the set of bipartitional relations on X by $\Bip(X)$.
\end{definition}

Note that, by definition, the complement of a bipartitional relation
is also a bipartitional relation. 
Following \cite{GNH}, we say that $x,y\in X$ are {\em
  incomparable}, if either both $(x,y)$ and $(y,x)$ belong to $U$, or
none of them does. We will use the notation $x\inc{U}y$ for such pairs.

\begin{lemma}[Han] \label{lem:Han}
The incomparability relation $\inc{U}$ is an equivalence relation.
\end{lemma}
This is \cite[Lemme~4]{GNH}, which may be easily verified directly, using
Definition~\ref{D_bip}.

As it was first observed by Han in \cite{GNH}, every bipartitional
relation $U$ induces a {\em linear order $<_U$} on the
$U$-incomparability classes as follows. For $x\not\inc{U}y$ we set
$x <_U y$ if and only if $(x,y)\in U$ but $(y,x)\not\in U$. 
The $U$-incomparability classes form a set partition of $X$ and we
may order them by $<_U$ to obtain an {\em ordered partition} of $X$. 
An ordered partition $(B_1,B_2,\ldots,B_k)$ of $X$ is an ordered list of
pairwise disjoint nonempty subsets $B_i\subset X$, such that $X$ is the
union of the sets $B_i$. Every
bipartitional relation may be represented by a unique pair of an
ordered partition $(B_1,B_2,\ldots,B_k)$
of $X$ and a vector $(\eps_1,\eps_ 2,\ldots,\eps_k)\in \{0,1\}^k$
(cf. \cite[Th.~5]{GNH}), as follows. We set
\begin{equation}
\label{E_obip}
(x,y)\in U \iff \left\{\begin{array}{l}
\mbox{$x\in B_i$ and $y\in B_j$ for some $i<j$,}\\
\quad \quad\quad\quad\quad\quad\quad\mbox{or}\\
\mbox{$x,y\in B_i$ for some $i$ satisfying $\eps_i=1$.}
\end{array}\right.
\end{equation}
In fact, the $B_i$'s must be the $\inc{U}$-equivalence classes,
numbered in such a way that $i<j$ if and only if $x<_U y$ for every $x\in B_i$
and $y\in B_j$. We must set $\eps _i=1$ if and only if $(x,x)\in U$ for
all $x\in B_i$.

For example, the bipartitional relation 
$U=\{(1,2),(1,3),(2,2),(2,3),(3,2),(3,3)\}$
has two $U$-equivalence classes: $\{1\}$ and
$\{2,3\}$. Since $1<_U 2$ and $1<_U 3$, we must have $B_1=\{1\}$ and
$B_2=\{2,3\}$. Moreover, $(1,1)\not\in U$ implies $\eps_1=0$, whereas
$(2,2)\in U$ and $(3,3)\in U$ imply $\eps_2=1$.

Following \cite{FoZeAE}, we call the ordered partition
$(B_1,B_2,\ldots,B_k)$ together with the vector $(\eps_1,\eps_
2,\ldots,\eps_k)$ an {\em ordered bipartition}, and we write it
as $(B_1^{\eps_1},B_2^{\eps_2},\ldots,B_k^{\eps_k})$. We call the
blocks $B_i$ satisfying $\eps_i=1$ {\em underlined\/}
(and, consequently, we call the blocks $B_i$ satisfying $\eps_i=0$ 
{\em nonunderlined}). 
Furthermore, we call the ordered
bipartition $(B_1^{\eps_1},B_2^{\eps_2},\ldots,B_k^{\eps_k})$ defining
$U$ via (\ref{E_obip}) the {\em ordered bipartition representation of
$U$}.
On the other hand,
every relation $U$ defined by an ordered bipartition representation
$(B_1^{\eps_1},B_2^{\eps_2},\ldots,B_k^{\eps_k})$ in the way above is
bipartitional: the transitivity of $U$ is clear, and the transitivity
of $(X\times X)\setminus U$ is evident from the following trivial
observation.

\begin{lemma}
\label{L_uc}
If $U\subseteq X\times X$ is represented by the ordered bipartition
$(B_1^{\eps_1},B_2^{\eps_2},\ldots,B_k^{\eps_k})$ of $X$, then 
$U^c:=(X\times
  X)\setminus U$ is represented by the ordered bipartition
  $(B_k^{1-\eps_{k-1}},B_{k-1}^{1-\eps_{k-1}},\break \ldots,B_1^{1-\eps_1})$.
\end{lemma}

We will use the notation\footnote{The letter $U$ has no specific
significance here, but we selected it in tribute to the ubiquitous letter
$U$ in Foata and Zeilberger's article \cite{FoZeAE}.}
$U(B_1^{\eps_1},B_2^{\eps_2},\ldots,B_k^{\eps_k})$ to denote 
the bipartitional relation defined by its ordered bipartition
representation $(B_1^{\eps_1},B_2^{\eps_2},\ldots,B_k^{\eps_k})$.
For example, the bipartitional relation 
$U=\{(1,2),(1,3),(2,2),(2,3),(3,2),(3,3)\}$
from above
may also be given as $U(\{1\}^0, \{2,3\}^1)$.

Frequently, we shall write this ordered bipartition in a suggestive
manner, where we physically underline the elements of underlined
blocks. For example, the above bipartitional relation will also be
written in the form $U(\{1\}, \{\underline 2,\underline 3\})$.

\section{Discrete Morse matching via chain enumeration}
\label{sec:Morse}

{\it Discrete Morse Theory}, developed by Forman~\cite{Forman-advmath,
  Forman-trans-ams, Forman-slc}, is a combinatorial theory that helps
to determine the homotopy type of a simplicial complex, considered as a 
$CW$-complex.  
Roughly speaking, in this theory a {\it Morse function} on the faces
of a simplicial complex induces a {\it Morse matching}, which in its turn
enables one to 
perform a sequence of elementary
collapses and find a smaller, homotopy equivalent $CW$-complex. Only the
unmatched faces of the simplicial complex ``survive'' the collapsing;
the subcomplexes induced by these faces are the {\it critical cells},
from which the homotopy type of the complex can (hopefully) be read off.
In our paper, we shall not need to know exact definitions of all these
ingredients. For our purpose it will suffice to keep
in mind that one of the primary goals is to identify the critical cells.
For a detailed
  description of the theory we refer the reader to the above cited
  sources.

In this paper we will use a method developed by Babson and
  Hersh~\cite{Babson-Hersh}, in the extended form of Hersh and Welker
\cite{Hersh-Welker} (which incorporates a correction to
\cite{Babson-Hersh} pointed out in \cite{Hersh-opt,SaVaAA}).
This method is
designed to find the homotopy type of the
  {\em order complex} $\triangle(P\setminus \{\widehat{0},\widehat{1}\})$ of a
  graded partially ordered set $P$ with minimum element $\widehat{0}$
  and maximum element $\widehat{1}$. Recall that the order complex of a
  partially ordered set $Q$ is the simplicial complex whose vertices are
  the elements of $Q$ and whose faces are the chains of $Q$. Babson and
  Hersh~\cite{Babson-Hersh} find a Morse matching on the Hasse
  diagram of the poset of faces of $\triangle(P\setminus
  \{\widehat{0},\widehat{1}\})$, the order relation being defined by
inclusion, by fixing 
  an enumeration of the maximal chains of $P$, which they call {\em poset
  lexicographic order}.
It was observed by Hersh and Welker~\cite[Theorem~3.1]{Hersh-Welker} 
that the key
property of a poset lexicographic order that is used in all 
proofs of Babson and Hersh in~\cite{Babson-Hersh} is that the
enumeration of maximal chains considered {\em grows by creating skipped
  intervals} (which is implicit in \cite[Remark~2.1]{Babson-Hersh}). 
They call this property the {\em crossing condition},
originally introduced by Hersh~\cite{Hersh-lexbal}.
The following definition is easily seen to be equivalent to this
crossing condition.
    
\begin{definition}
\label{D_ecsi}
Let $P$ be a graded poset of rank $n+1$ 
with a unique minimum element $\widehat 0$ 
and a unique maximum element $\widehat 1$.
An enumeration $c_1,\ldots,c_N$ of all maximal chains of $P$
{\em grows by creating
skipped intervals} if for every maximal chain $c_i$ there is a family of
intervals $I(c_i)$ with elements
$[a,b]=\{a,a+1,\ldots,b\}\subseteq \{1,2,\dots,n\}$, none of the intervals
contained in another, with the following property:
a chain $c$ contained in a maximal chain $c_i$ is also contained in 
a maximal chain $c_j$ for some $j<i$ 
if and only if the set of ranks of $c$ is
disjoint from at least one interval in $I(c_i)$.
\end{definition}

It is worth noting that the property stated in Definition~\ref{D_ecsi}
above also suffices to prove the linear inequalities shown
in~\cite{Billera-Hetyei-flag} and~\cite{Billera-Hetyei-planar}.

In the main result of Babson and Hersh \cite{Babson-Hersh}, a second
interval system, which is derived from the $I$-intervals, plays a crucial 
role. This interval system is called {\em $J$-intervals} $J(c_k)$.
The process of finding
the system of $J$-intervals is given in~\cite[p.~516]{Babson-Hersh} and
may be extended without any change to enumerations of maximal chains
that grow by creating skipped intervals as follows. 

\begin{definition}
\label{D_jint}
Consider an enumeration of all maximal chains of a graded poset of rank $n+1$
with $\widehat 0$ and $\widehat 1$
that grows by creating skipped intervals. Let $c_k$ be a maximal chain
whose associated interval system $I(c_k)$ satisfies 
$$\bigcup_{[u,v]\in
  I(c_k)}[u,v]=\{1,2,\dots,n\}.$$ 
We define the {\em associated $J$-intervals
  $J(c_k)$} as the output of the following process:

\begin{itemize}
\item[(0)]  Initialize by setting $I=I(c_k)$ and $J=\emptyset$.
\item[(1)] Let $[u,v]$ be the interval in $I$ whose left end point $u$
is the least. Add $[u,v]$ to $J$, and remove it from $I$.
\item[(2)] Replace each interval $[x,y]$ in $I$ by the intersection
$[x,y]\cap[v+1,n]$. Define the ``new" $I$ to be the 
resulting new family of intervals.
\item[(3)] Delete from $I$ those intervals which are not
  minimal with respect to inclusion.  
\item[(4)] Repeat steps {\em (1)--(3)} until $I=\emptyset$. 
The output of the algorithm is $J$.
\end{itemize}
\end{definition}

Our wording differs slightly from the one used by Babson and Hersh,
since they consider the families $I(c_k)$ and $J(c_k)$ as families of
subsets of $c_k$, whereas we consider them as families of subsets of
$\{1,\ldots, n\}$. 

The following theorem presents the main theorem of Babson and Hersh
\cite[Th.~2.2, Cor.~2.1]{Babson-Hersh},
in the generalized form implied by 
\cite[Theorem 3.1]{Hersh-Welker}
(including the aforementioned correction to \cite{Babson-Hersh}).

\begin{theorem}[Babson--Hersh]
\label{T_BH1}
Let $P$ be a graded poset of rank $n+1$ with $\widehat0$ and $\widehat1$,
and let $c_1,\ldots,c_N$ be an
enumeration of its maximal chains that grows by creating skipped
intervals. Then, in the Morse matching constructed by Babson and
Hersh in \cite[paragraphs above Th.~2.1]{Babson-Hersh}, 
each maximal chain $c_k$ contributes at most one critical
cell. The chain $c_k$ contributes a critical cell exactly when the union
of all intervals listed in $J(c_k)$ equals $\{1,2,\ldots,n\}$.
If a maximal chain $c_k$ contributes a critical cell, then the dimension
of this critical cell is one less than the number of intervals listed in
$J(c_k)$. 
\end{theorem} 

We will use the above result in combination with
the main theorem of Discrete Morse Theory due to Forman
\cite[first (unnumbered) corollary]{Forman-advmath}, 
\cite[Th.~0.1]{Forman-trans-ams}, \cite[Th.~2.5]{Forman-slc}.

\begin{theorem} \label{thm:Forman}
Suppose $\triangle$ is a simplicial complex with a discrete Morse
function. Then $\triangle$ is homotopy equivalent to a CW complex with
exactly one cell of dimension $p$ for each critical cell of dimension
$p$. In particular, if there is no critical cell 
then $\triangle$ is contractible.  
\end{theorem}

\begin{remark}
We point out that Babson and Hersh modify Forman's conventions
by including the empty face in the range of the Morse
function, see the second paragraph after Definition~1.1 in 
\cite{Babson-Hersh}. As a consequence, 
a vertex might be matched to the empty face, something which is impossible in
the setup of Forman. The term ``critical cell" is thus
slightly more restrictive in \cite{Babson-Hersh} than in
\cite{Forman-advmath,Forman-trans-ams,Forman-slc} in that such a
vertex would be a critical cell according to Forman but not according
to Babson and Hersh.
\end{remark}

\section{The lattice of bipartitional relations}
\label{sec:bipart}

In this section, we formally define the order relation on the set of
bipartitional relations, and we prove that the so defined poset is a
lattice (see Theorem~\ref{thm:lattice}). At the end of this section,
we record an auxiliary result concerning the lattice structure of 
$\Bip(X)$ in Lemma~\ref{L_shortp}, 
which will be needed later in Section~\ref{sec:pi-compatible} in the
proof of Lemma~\ref{L_opl}.

Let $U$ and $V$ be two bipartitional relations in $\Bip(X)$. We define
$U\le V$ if and only if $U\subseteq V$ as subsets of $X\times X$.
In this manner, $\Bip(X)$ becomes a partially ordered set.

\begin{theorem} \label{thm:lattice}
For any finite set $X$, the poset $\Bip(X)$ is a lattice.
\end{theorem}
\begin{proof}
By \cite[Prop.~3.3.1]{Stanley-EC1},
it is sufficient to show that every pair of bipartitional
relations has a join. This will be done in Proposition~\ref{P_join} below.
\end{proof}

We remind the reader that a pair $(x,y)$ belongs to the {\em
transitive closure} of a relation 
$W\subseteq X\times X$ if there exists a chain
$x_0,x_1,\ldots x_n\in X$ with $n>0$ such that $x_0=x$, $x_n=y$ and
$(x_i,x_{i+1})\in W$ for $i=0,1,\ldots,n-1$.

\begin{proposition}
\label{P_join}
For every $U,V\in\Bip(X)$ there exists a
smallest bipartitional relation with respect to inclusion containing
both $U$ and $V$, that is, a join $U\vee V$. This join is given by the
transitive closure of $U\cup V$.
\end{proposition}
\begin{proof}
Let $W$ denote the transitive closure of $U\cup V$.
Every bipartitional relation containing both $U$ and $V$ contains also
$W$ by transitivity. We only need to show that $W$ is bipartitional.
It is clearly transitive, only the transitivity of $(X\times X)\setminus
W$ remains to be seen.

Assume by way of contradiction that $(x,y)$ and $(y,z)$ belong to the
complement of $W$ but $(x,z)\in W$ for some $x,y,z\in X$. By the
definition of $W$, there exists a sequence $x_0,x_1,\ldots,x_n\in X$
such that $n>0$, $x_0=x$, $x_n=z$, and for every $i\in\{0,1,\ldots,n-1\}$ we
have $(x_i,x_{i+1})\in U$ or $(x_i,x_{i+1})\in V$. Without loss of
generality we may assume that we have $(x_{n-1},z)\in U$.
We cannot have $n=1$ since this implies $(x,z)\in U$, in contradiction
with $(x,y)\not\in U\subseteq W$, $(y,z)\not\in U\subseteq W$, and the
transitivity of $U^c$
(where, as before, $U^c$ denotes the complement $(X\times X)\setminus U$). 
By induction on $i$, we see that $(x,x_i)$
belongs to $W$, for $i=1,2,\dots,n-1$. In particular, we have
$(x,x_{n-1})\in W$. The pair $(x_{n-1},y)$
cannot belong to $U$, otherwise we have $(x_{n-1},y)\in W$ and, by the
transitivity of $W$, also $(x,y)\in W$. On the other hand,
by the transitivity of the relation
$U^c$, we obtain from $(x_{n-1},y)\not\in U$ and $(y,z)\not\in U$
that $(x_{n-1},z)\not\in U$, in contradiction with our assumption.
\end{proof}

We may represent any relation $R\subseteq X\times X$ as a directed graph
on the vertex set $X$ by drawing an edge $x\rightarrow y$ exactly when
$(x,y)\in R$.
If we represent $U\cup V$ as a directed graph,
we obtain that $(x,y)\in U\vee V$ if and only if there is a directed
path $x=x_0\rightarrow x_1\rightarrow \ldots \rightarrow x_m=y$ such
that each edge belongs to the graph representing $U\cup V$.
By the transitivity of $U$ and $V$,
a shortest such path is necessarily {\em $UV$-alternating} in the sense
that every second edge belongs to $U$, the other edges belonging to
$V$. There is no bound on the minimum length of such a shortest path, as is
shown in the following example.

\begin{example}
Let $X=\{1,2,\ldots,n\}$ and consider the bipartitional relation
$$U=U(\{\underline{n},\underline{n-1}\},
\{\underline{n-2},\underline{n-3}\},\ldots),$$
where each block has
two elements, except possibly for the rightmost block, which is a singleton
if $n$ is odd. Consider also
$$V=U(\{\underline{n}\},\{\underline{n-1},
\underline{n-2}\},\{\underline{n-3},\underline{n-4}\},\ldots),$$
where each
block has two elements, except for the leftmost block, which is
always a singleton, and possibly for the rightmost block which is a
singleton if $n$ is even. It is easy to verify that
$$U\vee V=
(\{\underline{1},\underline{2},\ldots,\underline{n}\}).$$
The shortest
$UV$-alternating path from $1$ to $n$ is $1\rightarrow
2\rightarrow \cdots\rightarrow n$, since $(i,j)\not\in U\cup V$ if
$j-i\geq 2$.
\end{example}

On the other hand, if only $(x,y)$ belongs to $U\vee
V$ but $(y,x)$ does not, then the shortest $UV$-alternating path from
$x$ to $y$ has length $1$.
\begin{lemma}
\label{L_shortp}
Let $U$ and $V$ be bipartitional relations on $X$. If for some $x,y\in
X$ we have $(x,y)\in U\vee V$ and $(y,x)\not\in U\vee V$ then $(x,y)$
already belongs to $U\cup V$.
\end{lemma}
\begin{proof}
Assume, by way of contradiction, that the shortest $UV$-alternating path
$x=x_0\rightarrow x_1\rightarrow \ldots \rightarrow x_m=y$ from $x$ to
$y$ satisfies
$m>1$. Then, because of $m>1$, $(x,y)$ belongs to $U^c$ and $V^c$. Since
$(y,x)\not\in U\vee V$, the pair $(y,x)$ also belongs to $U^c$ and $V^c$.
Thus $x\inc{U} y$ and $x\inc{V} y$. We claim that we may replace
$x_0=x$ with $y$ and $x_n=y$ with $x$ in the $UV$-alternating path
$x=x_0\rightarrow x_1\rightarrow \ldots \rightarrow  x_m=y$ and obtain 
a $UV$-alternating path
$y\rightarrow x_1\rightarrow \ldots \rightarrow x$.
Indeed, $x\inc{U} y$ and $(x,y)\not\in U$ imply that $x$ and $y$
belong to the same nonunderlined block of $U$. Hence, if
$(x_0,x_1)\in U$, then $x_1$ belongs to a
block of $U$ to the ``right" of the block containing $x$, whence
$(y,x_1)\in U$. Similarly, if
$(x_0,x_1)\in V$,
then $x\inc{V} y$ and $(x,y)\not\in V$ yield $(y,x_1)\in V$. The proof
that $x_n$ may be replaced with $x$ is analogous.
We obtain that there is a $UV$-alternating path from $y$ to $x$,
implying $(y,x)\in U\vee V$, in contradiction to our
assumption. Therefore we must have $m=1$.
\end{proof}

\section{Cover relations and rank function}
\label{sec:rank}

In this section we describe the cover relations in the bipartition
lattice $\Bip(X)$. This description will allow us to show that $\Bip(X)$
is a graded poset, and to give an explicit formula for the rank function.

\begin{theorem}
\label{T_cover}
Let $U,V\in \Bip(X)$ be bipartitional relations. Then $V$ covers $U$ if
and only if its ordered bipartition representation may be obtained
from the ordered bipartition representation of $U$ in one of the three
following ways:
\begin{enumerate}
\item[(i)] join two adjacent underlined blocks of $U$,
\item[(ii)] separate a nonunderlined block of $U$ into two adjacent
nonunderlined blocks, or
\item[(iii)] change a nonunderlined singleton block of $U$ into an underlined
singleton block.
\end{enumerate}
Moreover, $\Bip(X)$ is a graded poset, with rank function
\begin{equation}
\label{E_rank}
\rank (U(B_1^{\eps_1},B_2^{\eps_2},\ldots,B_k^{\eps_k}))=
   3\cdot \sum _{i:\eps _i=1} \left| B_i\right|
   +\left|\left\{i \::\: \eps _i=0\right\}\right|
   -\left|\left\{i \::\: \eps _i=1\right\}\right|-1.
\end{equation}
\end{theorem}

\begin{example}
The cover relations in $\Bip(\{1,2\})$ are represented in
Figure~\ref{fig:bip2}. The cover relations in a subset of
$\Bip(\{1,2,3\})$ are represented in Figure~\ref{fig:bip3c}. (The fact 
that the cover relations in the latter subposet are also cover relations
in the entire poset $\Bip(\{1,2,3\})$ is shown in
  Proposition~\ref{T_fullrank}.)  
\end{example}

\begin{proof}[Proof of Theorem~\ref{T_cover}]
First we show that the ordered bipartition representation of $V$ must
come from the ordered bipartition representation of $U$ in one of the
three ways mentioned in the statement. For that purpose, assume that $V$
covers $U=U(B_1^{\eps_1},B_2^{\eps_2},\ldots,B_k^{\eps_k})$.
Let us compare the
restrictions of $V$ and $U$ to every block $B_i$.
Note that the restriction of a bipartitional relation on $X$
to a subset of $X$ is also bipartitional.

\medskip
{\sc Case 1.} {\em $\rest {V}{B_i}$ properly contains $\rest {U}{B_i}$
for some $i$.} In this case we must have $\eps_i=0$. The relation
$W$ given by  $$(x,y)\in W\quad \text{if and only if}
\quad  \left\{\begin{array}{l}
\mbox{$(x,y)\in U$,}\\
\quad \quad\quad\quad \mbox{or}\\
\mbox{$x,y\in B_i$ and $(x,y)\in V$,}
\end{array}\right.
$$
is a bipartitional relation, properly containing $U$, and contained in
$V$. In fact, its ordered bipartition representation may be obtained
from $(B_1^{\eps_1},B_2^{\eps_2},\ldots,B_k^{\eps_k})$
by replacing $B_i^{\eps_i}=B_i^0$ with the ordered bipartition
representation of $\rest{V}{B_i}$. Since $V$ covers $U$, we must have
$V=W$.

If $\rest{V}{B_i}$ contains no underlined block then merging two adjacent
blocks of $\rest{V}{B_i}$  yields a bipartitional relation $U'$ on $B_i$
satisfying $\rest{U}{B_i}\subseteq U'\subsetneqq \rest{V}{B_i}$. Since
$V$ covers $U$ and, hence, $\rest{V}{B_i}$ covers $\rest{U}{B_i}$, we must
have $U'=\rest{U}{B_i}$.
Therefore $V$ is obtained from $U$ by an operation
of type (ii).

If $\rest{V}{B_i}$ contains an underlined block, then by changing this
block to nonunderlined we may obtain a 
bipartional relation properly contained in
$V$ and still containing $U$. Hence $\rest{V}{B_i}$ must be $B_i\times
B_i$. The only case when there is no bipartition on $B_i$ strictly
between $\emptyset$ and  $B_i\times B_i$ is when $|B_i|=1$, and $V$ is
obtained from $U$ by an operation of type (iii).

\medskip
{\sc Case 2.} {\em $\rest {V}{B_i}=\rest {U}{B_i}$ for all $i$.}
In this case every $\inc{U}$-equivalence class is contained in some
$\inc{V}$-equivalence class, and this containment is proper for at least
one of the $B_i$'s, since otherwise we must have $V=U$. Hence the
situation of Case~1 applies to at least one of the blocks of $V^c$
and $U^c$. (Clearly, $U^c$ must cover $V^c$). Thus, by the already
proven case, the ordered bipartition representation of $U^c$ must be
obtained from the ordered bipartition representation of $V^c$ by an
operation of type (ii) or (iii). Here we may exclude an operation of
type (iii), since we are not allowed to have the $\inc{U}$-equivalence
classes (which are the same as the $\inc{U^c}$-equivalence classes) to
coincide with the  $\inc{V}$-equivalence classes. Therefore $U^c$ is
obtained from $V^c$ by an operation of type (ii), which by 
Lemma~\ref{L_uc} is equivalent to saying that $V$ is obtained 
from $U$ by an operation of type (i).
\medskip

It is easy to see that the function $\rank$ given in (\ref{E_rank})
assigns zero to the empty bipartitional relation 
$U(X^0)$, and increases by exactly one every time we
perform one of the operations (i), (ii), or (iii). By the already
established
part of the statement, $\rank$ increases by one on every cover relation,
and so $\Bip(X)$ is a graded poset with rank function $\rank$. On the
other hand, every operation of type (i), (ii), or (iii) on a
bipartitional relation $U$ must yield a
bipartitional relation $V$ covering $U$, since the rank function has
increased by exactly one.
\end{proof}

\begin{corollary} \label{cor:rank}
If $|X|=n$ then $\Bip(X)$ has rank $3n-2$.
\end{corollary}

\section{$\pi$-compatible bipartitions}
\label{sec:pi-compatible}

The purpose of this section is to introduce the notion of
compatibility of bipartitional relations with a given ordered
partition (the latter having been defined in the paragraph after
Lemma~\ref{lem:Han}). This notion will be of crucial importance for the
subsequent structural analysis of $\Bip(X)$ in the subsequent
sections.
As a first application, we use it in Proposition~\ref{P_cont} 
to give a criterion to decide $U\subseteq V$ when
$U$ and $V$ are bipartitional relations given by their ordered
bipartition representations.

\begin{definition} \label{def:comp}
We call an ordered partition $\pi=(C_1,\ldots,C_k)$ {\em compatible} with
the bipartitional relation $U$, if for every $x,y\in X$ we have
$$x\in C_i, y\in C_j, (x,y)\in U, (y,x)\not\in U
\quad \text{imply}\quad  i<j.$$
\end{definition}
Equivalently, if $U=U(B_1^{\eps_1},B_2^{\eps_2},\ldots,B_l^{\eps_l})$,
then every $B_i$ is the union of consecutively indexed $C_j$'s. A
particular case arises if $\pi$ consists of singleton blocks only. In
this case, given that $X=\{x_1,x_2,\dots,x_n\}$,
there is a permutation $\rho$ of the elements of $X$ such
that $\pi=(\{\rho(x_1)\},\{\rho(x_2)\},\dots,\{\rho(x_n)\})$. 
By abuse of
terminology, we shall often say in this case that ``the ordered
partition $\pi$ 
{\it is a permutation}," and the bipartitional relation
$U$ is compatible with such an ordered partition $\pi$ if and only if the
elements of $B_1,B_2,\ldots,B_l$ may be listed in such an order that
placing these lists one after the other in increasing order of blocks
gives the left-to-right reading of the permutation $\pi$. For any
ordered partition $\pi$, we denote the subposet of $\pi$-compatible
bipartitions in $\Bip(X)$ by $\Bip_\pi(X)$. The Hasse diagram of
$\Bip_{(\{1\},\{2\},\{3\})}(\{1,2,3\})$ is shown in
Figure~\ref{fig:bip3c}. 

\begin{figure}[h]
\begin{center}
\input{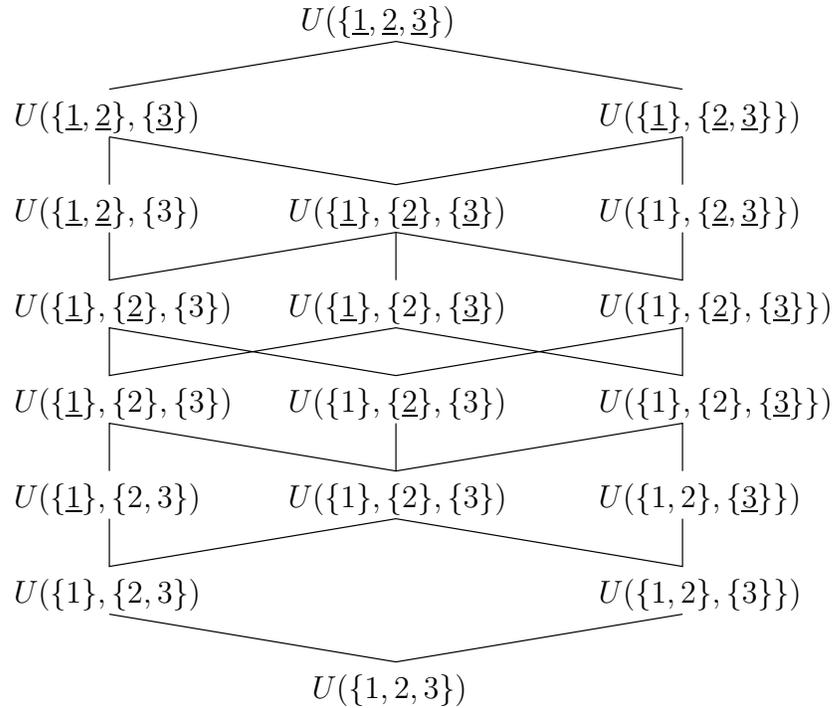}
\end{center}
\caption{$\Bip_{(\{1\},\{2\},\{3\})}(\{1,2,3\})$}
\label{fig:bip3c}
\end{figure}

The next lemma shows that this subposet is also a
sublattice.

\begin{lemma}
\label{L_opl}
Let $\pi$ be an ordered partition of $X$.
If $U\subseteq X\times X$ and $V\subseteq X\times X$ are
$\pi$-compatible bipartitional relations then so are
$U\wedge V$ and $U\vee V$.
\end{lemma}
\begin{proof}
Let $\pi=(C_1,\ldots,C_k)$ and assume $(x,y)\in U\vee V$ but
$(y,x)\not\in U\vee V$ for some $x\in C_i$ and $y\in C_j$.
By Lemma~\ref{L_shortp}, we have $(x,y)\in U\cup
V$. Without loss of generality we may assume $(x,y)\in U$. Since $U$ is
$\pi$-compatible, we obtain $i<j$. Hence $U\vee V$ is also
$\pi$-compatible. The other half of the statement follows by duality,
since any bipartitional relation is $(C_1,\ldots,C_k)$-compatible if
and only if its complement is $(C_k,\ldots,C_1)$-compatible.
\end{proof}

Using Theorem~\ref{T_cover} we may deduce the following fact.

\begin{proposition}
\label{P_c}
Let $c: \emptyset=U_0\subset U_1\subset \cdots \subset U_{3n-2}=X\times X$ be a
maximal chain in $\Bip(X)$, where $n=|X|$. Then there is a unique
ordered partition $\pi_c$ which is compatible with all elements of the
chain. This ordered partition is a permutation.
\end{proposition}
\begin{proof}
For $n=1$ the statement is trivially true. Assume $n\geq 2$ and let $x$
and $y$ be two different elements of $X$. Consider the smallest $i$ for
which $U_i$ contains at least one of $(x,y)$ and $(y,x)$. Such an $i$
exists since $U_{3n-2}=X\times X$, and it is positive since $U_0=\emptyset$.
We claim that exactly one of $(x,y)$ and $(y,x)$ will belong to $U_i$.
In fact, $U_{i-1}$ does not contain any of them, so $x$ and $y$ belong
to the same nonunderlined $\inc{U_{i-1}}$-equivalence class. $U_i$ is obtained
from $U_{i-1}$ by one of the operations described in 
Theorem~\ref{T_cover}. 
Since at least one of $(x,y)$ and $(y,x)$ was added, this
operation can only be the separation of the $\inc{U_{i-1}}$-equivalence
class of $x$ and $y$ into two nonunderlined blocks. Such an operation
adds exactly one of $(x,y)$ and $(y,x)$. Let us set $x<_c y$ if
$(x,y)\in U_i$ and $(y,x)\not\in U_i$,
respectively $y<_c x$ if $(y,x)\in U_i$ and $(x,y)\not\in U_i$.

We want to construct an ordered partition $\pi_c$ which is compatible 
with all $U_i$'s. If $x<_c y$, this implies that $x$
belongs to an earlier block of 
$\pi_c$ than $y$. There is at most one such
ordered partition: the permutation, induced by the relation $<_c$,
provided that $<_c$ is a linear order.

We are left to show that $<_c$ is indeed a linear order.
Clearly, for distinct $x$ and $y$
exactly one of $x<_c y$ and $y<_c x$ holds. We only need to
show the transitivity of the relation $<_c$. Assume by way of
contradiction that $x<_c y$, $y<_c z$ and $z<_c x$ hold for some
$\{x,y,z\}\subseteq X$. Then we have
$$\begin{matrix}
(x,y)\in U_i, & (y,x)\notin U_i,\\
(y,z)\in U_j, & (z,y)\notin U_j,\\
(z,x)\in U_k, & (x,z)\notin U_k,
\end{matrix}
$$
for some $i,j,k$. By the cyclic symmetry of the list $(x,y,z)$ we may
assume that either $i\leq j\leq k$ or $k\le j\le i$.

If  $i\leq j\leq k$, then, since
$(x,y)\in U_i\subseteq U_j$ and $(y,z)\in U_j$,
the transitivity of the relation $U_j$ implies $(x,z)\in
U_j\subseteq U_k$, which is in contradiction with
$(x,z)\notin U_k$.

On the other hand, if $k\le j\le i$, then since
$(y,z)\in U_j$ and $(z,x)\in U_k\subseteq U_j$,
the transitivity of the relation $U_j$ implies $(y,x)\in
U_j\subseteq U_i$, which is in contradiction with
$(y,x)\notin U_i$.
\end{proof}

Proposition~\ref{P_c} allows us to characterize $U\subseteq V$ when
$U$ and $V$ are bipartitional relations given by their ordered
bipartition representation.

\begin{proposition}
\label{P_cont}
Let $U,V\in \Bip(X)$ be bipartitional relations represented as
$U=U(B_1^{\eps_1}, B_2^{\eps_2},\ldots, B_k^{\eps_k})$
and $V=U(C_1^{\eta_1}, C_2^{\eta_2},\ldots, C_l^{\eta_l})$. Then $U$ is
contained in $V$ if and only if the following three conditions are satisfied:

\begin{enumerate}
\item[(i)] there is an ordered partition
$\pi=(\{\pi_1\},\{\pi_2\},\ldots,\{\pi_n\})$
   that is also a permutation
which is compatible with both $U$ and $V$,
\item[(ii)] every underlined $B_i$ is contained in some underlined $C_j$,
\item[(iii)] every nonunderlined $C_i$ is contained in some
   nonunderlined $B_j$.
\end{enumerate}
\end{proposition}

\begin{proof}
Assume first that $U$ is contained in $V$. Then there is a maximal
chain $c$ in $\Bip(X)$ containing both $U$ and $V$. By
Proposition~\ref{P_c} there is an ordered partition $\pi_c$ compatible
with every
element of $c$, and this ordered partition is a permutation, so
condition (i) is satisfied. Consider an underlined
block $B_i$. For every $x,y\in B_i$ we have $(x,y)\in
U$ and so $(x,y)\in V$ since $U\subseteq V$. Hence $B_i$ is contained in
some $C_j$. The proof of condition (iii) is analogous.

We are left to show that whenever $U$ is not contained in $V$, at least
one of the given conditions is violated. Assume $U\not\subseteq V$ and
consider an ordered pair $(x,y)\in U\setminus V$. If $(y,x)\in U$ holds
as well then $x$ and $y$ are contained in the same underlined block in
the representation of $U$. Thus condition (ii) is violated since
$(x,y)\not\in V$. Similarly $(y,x)\not\in V$ implies a violation of
condition (iii). We are left with the case where $(x,y)\in U$,
$(y,x)\not\in U$, $(x,y)\not\in V$, and $(y,x)\in V$. Now condition (i)
is violated. Indeed, let $\pi=(\{\pi_1\},\{\pi_2\},\ldots,\{\pi_n\})$
be an arbitrary ordered partition that is also a permutation, satisfying
$x=\pi_i$ and $y=\pi_j$. By definition, if $\pi$ is compatible
with $U$ then we must have $i<j$ while compatibility with $V$
requires just the opposite, $j<i$.
\end{proof}

\section{The distributivity of the sublattice of $\pi$-compatible
bipartitions} 
\label{sec:distr}

In this section we introduce a representation of all $\pi$-compatible
bipartitions, where $\pi$ is an arbitrary fixed permutation. We will use
this representation to show that $\Bip_\pi(X)$ is a distributive
lattice, for all ordered partitions $\pi$. Without loss of generality, we
may assume $X=\{1,2,\dots,n\}$ and, for the moment, we may even assume that
$\pi=(\{1\},\{2\},\ldots,\{n\})$. The analogous results for an arbitrary
finite set $X$ and an arbitrary permutation $\pi$ may be obtained by
renaming the elements. 

\begin{definition}
Let $U$ be a
$(\{1\},\{2\},\ldots,\{n\})$-compatible
bipartitional relation, represented as $U=U(B_1^{\eps_1},
B_2^{\eps_2},\ldots, B_k^{\eps_k})$, such that the elements in each block
are listed in increasing order. We define the {\em code of $U$} as
the vector $(u_1,\ldots, u_n)$ where each $u_i$ is an element of the set
$\{\pm 1, \pm 3\}$, given by the following rule:
$$
u_i=
\begin{cases}
-1&\mbox{if $i$ is listed as the first element in a nonunderlined block;}\\
-3&\mbox{if $i$ is in a nonunderlined block, but not listed first;}\\
\hphantom{-}1&\mbox{if $i$ is listed as the last element in an
underlined block;}\\
\hphantom{-}3&\mbox{if $i$ is in an underlined block, but not listed last.}\\
\end{cases}
$$
\end{definition}
For example, the code of the bipartitional relation
$U(\{\underline{1},\underline{2}\},\{3\},\{4,5\},\{\underline{6}\})$
is $(3,1,\break -1,-1,-3,1)$. Evidently, the ordered bipartition representing
$U$ may be uniquely reconstructed from its code, we only need to
determine which vectors are valid codes of bipartitional relations.

The definition of the code of $U$ is inspired by formula (\ref{E_rank})
giving the rank of $U$. According to this formula, we may compute
$\rank(U)$ of a $(\{1\},\{2\},\ldots,\{n\})$-compatible 
bipartional relation $U$ as
follows. We take the ordered bipartition representation of $U$, where we
list the elements in increasing order. For the first element in each
nonunderlined block we increase $\rank(U)$ by $1$, and we associate 
no contribution to the other elements in nonunderlined blocks. For the last
element in each underlined block we increase $\rank(U)$ by $2$, and
for each other element of an underlined block we increase $\rank(U)$ by
$3$. Thus we could equivalently define a code where the ordered list of
weights $(-1,-3,1,3)$ is replaced by the list $(1,0,2,3)$. The rank of $U$ is
the sum of the coordinates in this ``simpler code.'' The list of weights
$(-1,-3,1,3)$ is obtained from $(1,0,2,3)$ by the linear transformation 
$x\mapsto 2x-3$. Thus, even for the code we have chosen, $\rank(U)$ is a
linear function of the sum of the coordinates in its code. Our choice of
code has two ``advantages'' over the ``more obvious'' code described
above:
\begin{itemize}
\item[--] The description of a valid code in Corollary~\ref{C_vcode} below
  involves very simple linear inequalities with integer bounds.
\item[--] For our code, the code of $U^c$ is obtained by simply taking the
  negative of the code of $U$. 
\end{itemize}
In the end, it is only a matter of taste whether one
prefers the list of weights\break $(-1,-3,1,3)$ or the list $(1,0,2,3)$, and
the results below may be easily transformed to fit the reader's
preference. 

\begin{lemma}
\label{L_vcode}
A vector $(u_1,\ldots, u_n)\in\{\pm 1,\pm3\}^n$ is the code of a
$(\{1\},\{2\},\ldots,\{n\})$-compatible
bipartitional relation if and
only if the following conditions are satisfied:
\begin{itemize}
\item[(i)] $u_1\neq -3$;
\item[(ii)] $u_n\neq 3$;
\item[(iii)] if $u_i=-3$ for some $i>1$ then $u_{i-1}<0$;
\item[(iv)] if $u_i=3$ for some $i<n$ then $u_{i+1}>0$.
\end{itemize}
\end{lemma}
\begin{proof}
The necessity of the conditions above is obvious.

Conversely, given a vector $(u_1,\ldots, u_n)\in\{\pm 1,\pm 3\}^n$
satisfying the conditions above, we may find a unique ordered bipartition
$(B_1^{\eps_1},B_2^{\eps_2},\ldots, B_k^{\eps_k})$ representing a
relation whose code is $(u_1,\ldots, u_n)$, as
follows:
\begin{itemize}
\item[(a)] Start the first block with $1$ if $u_1=-1$ and with
   $\underline{1}$ if $u_1>0$. Continue reading the $u_i$'s, left to
   right.
\item[(b)] For $1<i<n$, if $u_i=-1$, start a new nonunderlined block with
   $i$. Note that rule (iv) prevents us from starting a nonunderlined
   block without ending a preceding underlined block.
\item[(c)] For $1<i\leq n$, if $u_i=-3$ then add a nonunderlined $i$ to the
   nonunderlined block that is currently being written (by condition (iii)).
\item[(d)] For $1<i\leq n$, if $u_i=1$ then end an underlined block with
   $\underline{i}$. This block is a singleton if $u_{i-1}<0$, and so
   $i-1$ belongs to a preceding nonunderlined block, or if $u_{i-1}=1$, and
   so $\underline{i-1}$ ends the preceding underlined block.
\item[(e)] For $1<i<n$, if $u_i=3$, then add an underlined
   $\underline{i}$ to the current underlined block if $u_{i-1}=3$, and
   start a new underlined block with $\underline{i}$ if $u_{i-1}<3$.
\end{itemize}
Clearly the above process yields the only $U$ whose code is
$(u_1,\ldots, u_n)$, and conditions (i) through (iv) guarantee that the
process never halts with an error.
\end{proof}

Lemma~\ref{L_vcode} may be rephrased in terms of inequalities as follows.
\begin{corollary}
\label{C_vcode}
A vector $(u_1,\ldots, u_n)\in\{\pm 1,\pm3\}^n$ is the code of a
$(\{1\},\{2\},\ldots,\{n\})$-compatible
bipartitional relation if and
only if it satisfies $u_1\geq -1$, $u_n\leq 1$ and $u_i-u_{i+1}\leq 2$
for $i=1,2,\ldots,n-1$.
\end{corollary}

\begin{theorem}
\label{T_repr}
Let $U$ and $V$ be
$(\{1\},\{2\},\ldots,\{n\})$-compatible
bipartitional
relations with codes $(u_1,\ldots, u_n)$ respectively $(v_1,\ldots, v_n)$.
Then $U\subseteq V$ if and only if $u_s\leq v_s$ holds
for $s=1,2,\ldots,n$.
\end{theorem}
\begin{proof}
Assume that $U=U(B_1^{\eps_1}, B_2^{\eps_2},\ldots,
B_k^{\eps_k})$ and $V=U(C_1^{\eta_1}, C_2^{\eta_2},\ldots, C_l^{\eta_l})$.
Since $U$ and $V$ are both
$(\{1\},\{2\},\ldots,\{n\})$-compatible, by Proposition~\ref{P_cont},
$U$ is contained in $V$ if and only if every underlined $B_i$ is
contained in some underlined $C_j$ and every nonunderlined $C_i$ is
contained in some nonunderlined $B_j$. It suffices to show that this is
equivalent to $u_s\leq v_s$ for all $s$.

Assume $U\subseteq V$ first, and consider the possible values of $u_s$,
for a fixed $s\in\{1,2,\ldots,\break n\}$. Let $B_j$ be the block of $U$
containing $s$.
If $u_s=-3$, then $u_s\leq v_s$ is automatically true. If $u_s=-1$ then
$v_s$ cannot be $-3$, otherwise the nonunderlined block $C_i$
containing $s$ has a smaller element in $C_i$, 
whereas the least
element of the nonunderlined block $B_j$ is $s$. Only $B_j$ could
contain $C_i$, but it does not. This contradiction shows that $v_s\geq
-1=u_s$. If $u_s=1$ then $s$ is an element in an underlined block $B_i$ of
$U$. This block $B_i$ must be contained in some underlined $C_j$.
In other words, $s$ belongs to an underlined block in $V$ 
showing $v_s\geq 1=u_s$. Finally, if $u_s=3$, then
$\{\underline{s}, \underline{s+1}\}$ is the subset of some underlined
$B_i$. This $B_i$ is contained in some underlined $C_j$, for which
we must have $\{\underline{s}, \underline{s+1}\}\subseteq C_j$. 
Thus, $s$ is not the last element in $C_j$, forcing $v_s\geq 3=u_s$.

For the converse, assume, by way of contradiction, that $u_s\le v_s$
for $1\le s\le n$,  but $U\not\subseteq V$. Then either condition (ii)
or (iii) of Proposition~\ref{P_cont} is violated.

\medskip
{\sc Case 1.} Some nonunderlined $C_i$ is not contained in any
nonunderlined $B_j$. If the least element $s$ of $C_i$
belongs to some underlined $B_j$ then, 
because of $v_s<0$ and $u_s>0$ we have
$u_s>v_s$, a contradiction. 
It remains the case where $s$ belongs to some
nonunderlined $B_j$. In this case let $t$ be the least element
of $C_i$ which does not belong to the same $B_j$ as $s$. Such a $t$
exists since the entire block $C_i$ is not contained in $B_j$.
Now we have $v_t=-3$ and $u_t\geq -1$, implying $v_t<u_t$, again
contradicting our assumption.

\medskip
{\sc Case 2.} Some underlined $B_i$ is not contained in any
underlined $C_j$. This case is the dual of the previous one, see also
Lemma~\ref{L_dual} below.

In both cases we obtain that $U\not\subseteq V$ implies $u_t>v_t$ for
some $t$, which is absurd.
\end{proof}

\begin{lemma}
\label{L_dual}
If $(u_1,\ldots,u_n)$ is the code of the bipartitional relation $U$ then
$(-u_n,\ldots,\break -u_1)$ is the code of $U^c.$
\end{lemma}
The proof is straightforward and is left to the reader.

Using Theorem~\ref{T_repr} and Corollary~\ref{C_vcode} we are able to
show that the sublattice of $\pi$-compatible
bipartitional relations is distributive for {\it any} ordered
partition $\pi$.

\begin{theorem}
\label{T_dist}
Let $\pi$ be any ordered partition on $X$. Then the lattice
$\Bip_\pi(X)$ is distributive.
\end{theorem}
\begin{proof}
Without loss of generality, let $X=\{1,2,\dots,n\}$.
It suffices to consider the case where 
$\pi=(\{1\},\{2\},\dots,\{n\})$.
For, if $\pi=(C_1,C_2,\dots,C_k)$, then, from the remarks immediately
following Definition~\ref{def:comp}, it is easy to see by compressing
the blocks $C_i$ to singletons $\{i\}$ that 
\begin{equation} \label{eq:compress} 
\Bip_\pi(\{1,2,\dots,n\})\cong
\Bip_{(\{1\},\{2\},\dots,\{k\})}(\{1,2,\dots,k\}).
\end{equation} 

From now on, let $\pi=(\{1\},\{2\},\dots,\{n\})$.
By Theorem~\ref{T_repr} and
Corollary~\ref{C_vcode}, the partially ordered set $\Bip_\pi(X)$ is
isomorphic to the set of all vectors $(u_1,\ldots,u_n)\in\{\pm 1, \pm
3\}$ satisfying $u_1\geq -1$, $u_n\leq 1$ and $u_i-u_{i+1}\leq 2$ for
$i=1,2,\ldots,n-1$, partially ordered by the relation
$(u_1,\ldots,u_n)\leq (v_1,\ldots,v_n)$ if and only if $u_i\leq v_i$
holds for all $i$. We claim that the join and meet operations in this
representation are given by
$$
(u_1,\ldots,u_n)\vee
(v_1,\ldots,v_n)=(\max(u_1,v_1),\ldots,\max(u_n,v_n))
\quad
\mbox{and}$$
$$
(u_1,\ldots,u_n)\wedge
(v_1,\ldots,v_n)=(\min(u_1,v_1),\ldots,\min(u_n,v_n)).
$$
Clearly the above operations yield the join and meet of the two vectors
in the larger lattice of {\it all\/} vectors from $\{\pm 1,\pm 3\}^n$, partially
ordered by the Cartesian product of natural orders of integers. Thus we
only need to show that $(\max(u_1,v_1),\ldots,\max(u_n,v_n))$ and
$(\min(u_1,v_1),\ldots,\min(u_n,v_n))$ satisfy the inequalities required
by Corollary~\ref{C_vcode}, given that $(u_1,\ldots,u_n)$ and $(v_1,\ldots,v_n)$
satisfy these inequalities. The verification of this observation is
straightforward and is left to the reader. The theorem now follows from the
fact that the $\max$ and $\min$ operations are distributive over each
other.
\end{proof}

The next example shows that the entire lattice
$\Bip(X)$ is not distributive for $|X|\geq 2$, and that, in fact, 
it is not even modular.

\begin{example} \label{ex:modular}
Let $X=\{1,2,\ldots,n\}$ for some $n\geq 2$ and consider the
bipartitional relations $U_1=U(\{1\},\{2\},\ldots,\{n\})$, 
$U_2=U(\{\underline{1}\},\{\underline{2}\},\ldots,\{\underline{n}\})$,
and $V=U(\{n\},\{n-1\},\ldots,\{1\})$. It is easy to verify that
$U_1$ is contained in $U_2$, the join $U_1\vee V$ is $X\times X$, and
the meet $U_2\wedge V$ is $\emptyset$. The set
$\{U_1,U_2,V,\emptyset, X\times X\}$, shown in Figure~\ref{fig:nmod}, is
thus a sublattice, isomorphic to 
the smallest example of a nonmodular lattice.
\end{example}

\begin{figure}[h]
\begin{center}
\input{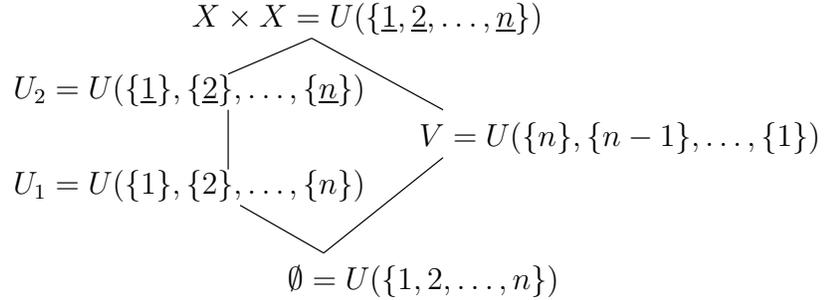}
\end{center}
\caption{Nonmodular sublattice contained in $\Bip(\{1,2,\ldots,n\})$}
\label{fig:nmod}
\end{figure}

By Birkhoff's theorem \cite[Th.~3.4.1]{Stanley-EC1}, 
every distributive lattice is isomorphic to
the lattice of order ideals in the poset of its join-irreducible
elements. In order to apply this result, we need to find the
join-irreducible elements in $\Bip_\pi(X)$.  
In preparation for the corresponding result (see
Theorem~\ref{P_jirred} below), we first characterize the cover relations in
$\Bip_\pi(X)$. Again, without loss of generality, we may assume that
$X=\{1,2,\ldots,n\}$ and $\pi=(\{1\},\{2\},\ldots,\{n\})$.

\begin{proposition}
\label{T_fullrank}
Let $X=\{1,2,\dots,n\}$, and let $\pi$ be the permutation
$\pi=(\{1\},\{2\},\break \ldots, \{n\})$.
Let $U\subsetneqq V$ be $\pi$-compatible bipartitional relations in
$\Bip(X)$. Then $V$ covers $U$ in $\Bip(X)$ if and only if\/
$V$ covers $U$ in $\Bip_\pi(X)$.
\end{proposition}
\begin{proof}
Clearly, if $V$ covers $U$ in $\Bip(X)$ it also covers it in
$\Bip_\pi(X)$. We only need to show that whenever $V\supset U$ holds
in $\Bip_\pi(X)$, then there is a $\pi$-compatible
$U'$ covering $U$ in $\Bip(X)$ such that
$U\subset U'\subseteq V$ holds in $\Bip_\pi(X)$.
We prove this statement by considering the codes
$(u_1,\ldots,u_n)$ of $U$ and $(v_1,\ldots,v_n)$ of $V$.
Assume that $u_i\leq v_i$ holds for
$i=1,2,\ldots,n$ and that $j$ is the least index such that $u_j<v_j$.

\medskip
{\sc Case 1.} $u_j=-3$. In this case $j>1$ and $u_{j-1}$
is negative. The element $j$ is in a nonunderlined block of $U$, 
and it is not the first element of this block. Let $U'$ be the
$(\{1\},\{2\},\ldots,\{n\})$-compatible
bipartitional relation obtained
from $U$ by splitting the block containing $j$ into two adjacent blocks,
such that the second block begins with $j$. Then the code $(u_1',\ldots,
u_n')$ of $U'$ is obtained from the code $(u_1,\ldots, u_n)$ by
increasing $u_j$ to $u_j'=-1$ and leaving all other coordinates
unchanged. Since $u_j<v_j$, we have $-1\leq v_j$, and thus $u_i'\leq v_i$
holds for $i=1,2,\dots,n$.

\medskip
{\sc Case 2.} $u_j=-1$. In this case $u_{j-1}$ (if it exists)
is not $3$ and $U$ has a nonunderlined block starting at $j$.
Since $-1=u_j<v_j$, we also have
$1\leq v_j$, and so $v_j$ is positive. 

\smallskip
{\it Subcase 2a.} $u_{j+1}=-3$, i.e., $j+1$ belongs to the
nonunderlined block of $U$ that started at $j$. Thus, by condition (iii) in
Lemma~\ref{L_vcode}, $v_{j+1}$ cannot be $-3$ and so $-1\leq v_{j+1}$.
Let $U'$ be the
$(\{1\},\{2\},\ldots,\{n\})$-compatible
bipartitional
relation obtained from $U$ by splitting the block containing $j+1$ into
two adjacent blocks, such that the second block begins with $j+1$. Just
like in Case~1, the code  $(u_1',\ldots,
u_n')$ of $U'$ is obtained from the code $(u_1,\ldots, u_n)$ by
increasing $u_{j+1}=-3$ to $u_{j+1}'=-1$, and so $u_i'\leq v_i$ holds for
all $i$.

\smallskip
{\it Subcase 2b.} $u_{j+1}\neq -3$, i.e., the nonunderlined
block containing $j$ is a singleton block. Let $U'$ be the
$(\{1\},\{2\},\ldots,\{n\})$-compatible
bipartitional relation obtained
from $U$ by changing the nonunderlined block $\{j\}$ into an underlined
block $\{\underline{j}\}$. The code  $(u_1',\ldots,
u_n')$ of $U'$ is obtained from the code $(u_1,\ldots, u_n)$ by
increasing $u_{j}=-1$ to $u_{j}'=1$, leaving all other coordinates
unchanged. Since $v_j\ge1$, we have $u_i'\le v_i$ for all $i$.

\medskip
{\sc Case 3.} $u_j=1$ (thus $v_j=3$). There is an underlined
block in $U$ ending with $j$, whereas the underlined block
containing $j$ in $V$ does not end with $j$.

\smallskip
{\it Subcase 3a.} $u_{j+1}<0$. By condition (iii) of
Lemma~\ref{L_vcode}, we must have $u_{j+1}=-1$. Moreover, $v_j=3$ and
condition (iv) of Lemma~\ref{L_vcode} imply $v_{j+1}\geq 1$.
Therefore there is a nonunderlined block in $U$ starting at $j+1$ and,
thus,
$u_{j+1}<v_{j+1}$ is satisfied. We may now repeat the reasoning of Case~2
for $j+1$.

\smallskip
{\it Subcase 3b.} $u_{j+1}>0$, i.e., there is an adjacent
underlined block in $U$ starting with $j+1$. Let $U'$ be the
$(\{1\},\{2\},\ldots,\{n\})$-compatible
bipartitional relation obtained
from $U$ by merging the underlined blocks containing $j$ and
$j+1$. The code $(u_1',\ldots,
u_n')$ of $U'$ is obtained from the code $(u_1,\ldots, u_n)$ by
increasing $u_{j}=1$ to $u_{j}'=3$, leaving all other coordinates
unchanged. Since $v_j=3$, we have $u_i'\le v_i$ for all $i$.
\end{proof}

\begin{corollary} \label{cor:samerank}
For any permutation $\pi$ of $X$, every interval $[U,V]$ in $\Bip_\pi(X)$ has
the same rank as the corresponding interval $[U,V]$ in $\Bip(X)$.
\end{corollary}

Now we are in the position to describe the join-irreducible elements in
$\Bip_\pi(X)$. 

\begin{theorem}
\label{P_jirred}
Let $X=\{1,2,\dots,n\}$ and $\pi=(\{1\},\{2\},\ldots,\{n\})$.
Then $\Bip_\pi(X)$ has the following $3n-2$ join-irreducible elements:
\begin{itemize}
\item[(i)] $E(i):=U(\{1,\ldots,i-1\},\{i,\ldots,n\})$ for
   $i\in\{2,\ldots,n\}$.
\item[(ii)] $F(i):=U(\{1,\ldots,i-1\},\{\underline{i}\},\{i+1,\ldots,n\})$ for
   $i\in\{1,2,\dots,n\}$. Here the first block is omitted for $i=1$ and
   the last block is omitted for $i=n$.
\item[(iii)]
   $G(i):=U(\{1,\ldots,i-1\},\{\underline{i},\underline{i+1}\},\{i+2,\ldots,n\})$
   for $i\in\{1,2,\dots,n-1\}$. Here the first block is omitted for $i=1$
   and the last block is omitted for $i=n-1$.
\end{itemize}
Moreover, the bipartitional relations listed under {\em(i)} and
{\em(ii)} are also
join-irreducible elements in $\Bip(X)$.
\end{theorem}
\begin{proof}
The bipartitional relations of type (i) above have rank $1$ and are
clearly join-irreducible elements even in the larger lattice $\Bip(X)$. By
Theorem~\ref{T_cover}, a bipartitional relation of type (ii) covers
exactly one element of $\Bip(X)$, namely
$U(\{1,\ldots,i-1\},\{i\},\{i+1,\ldots,n\})$. This
bipartitional relation belongs of course also to $\Bip_\pi(X)$. Thus
the bipartitional relations listed under (ii) are again
join-irreducible elements even in the larger lattice $\Bip(X)$.

The element $G(i)$ in (iii) is not join-irreducible in $\Bip(X)$ since,
by Theorem~\ref{T_cover}, it covers exactly the two elements 
$$
U(\{1,2,\dots,i-1\},\{\underline
i\},\{\underline{i+1}\},\{i+2,\dots,n\})
$$ 
and
$$
U(\{1,2,\dots,i-1\},\{\underline{i+1}\},\{\underline
i\},\{i+2,\dots,n\}).
$$ 
However, only one, namely the former, is $\pi$-compatible. Hence, by
Proposition~\ref{T_fullrank}, $G(i)$ covers exactly one element in
$\Bip_\pi(X)$, which means that it is join-irreducible in
$\Bip_\pi(X)$. 

Conversely, if $V$ is join-irreducible in $\Bip_\pi(X)$, then, be
definition, it covers exactly one element, $U$ say, in
$\Bip_\pi(X)$. By Proposition~\ref{T_fullrank}, $V$ covers $U$ also in
$\Bip(X)$. ($V$ may cover other elements in $\Bip(X)$ as well,
but they must not be $\pi$-compatible.)
By Theorem~\ref{T_cover}, $U$ can be obtained from $V$ by either
splitting an underlined block into adjacent underlined blocks, or
by joining two adjacent nonunderlined blocks, or by changing a singleton
underlined block into a nonunderlined block. $V$ is join-irreducible
if and only if exactly one such operation yields a $\pi$-compatible
$U$. This excludes the possibility of $V$ having two underlined
blocks, or an underlined block with more than two elements, or three
adjacent nonunderlined blocks. Furthermore, it also excludes the
possibility of $V$ having an underlined block at the same time as
having two adjacent nonunderlined blocks. It is now obvious that only
the possibilities listed under (i)--(iii) remain.
\end{proof}

Figure~\ref{fig:bipji} indicates the Hasse diagram of the subposet of
join-irreducible elements of
$\Bip_{(\{1\},\{2\},\ldots,\{n\})}(\{1,\dots,n\})$.

\begin{figure}[h]
\begin{center}
\input{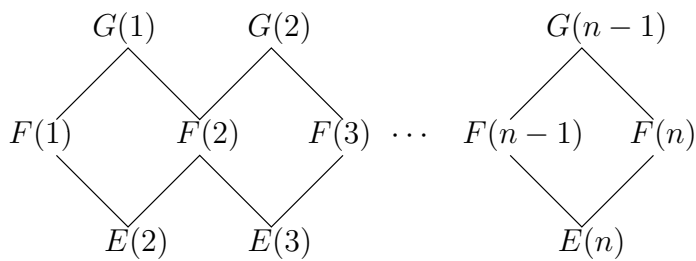}
\end{center}
\caption{Join-irreducible elements of 
$\Bip_{(\{1\},\{2\},\ldots,\{n\})}(\{1,\dots,n\})$}
\label{fig:bipji}
\end{figure}

\section{The J--T decomposition of the
order complex of an interval in the bipartition lattice}
\label{sec:JT}

This section contains preparatory material for the proofs of our main
results in Sections~\ref{s_toc} and \ref{sec:int}. The ultimate goal is
to construct an enumeration of all maximal chains in $\Bip(X)$,
respectively in any interval thereof, such that the results of Babson
and Hersh reviewed in Section~\ref{sec:Morse} become applicable. 
The way that we propose here to arrive there proceeds in two
steps. Recall that, by Proposition~\ref{P_c}, each maximal chain $c$
determines a unique permutation $\pi$ such that all elements of $c$
are $\pi$-compatible. The first step, performed in this section, 
will consist of finding a suitable
enumeration of all permutations. This induces a ``pre-enumeration" of
the maximal chains, by putting them together in smaller groups
according to their associated permutations and enumerating these
groups. Then, in the subsequent sections, we shall refine
this pre-enumeration further to a full enumeration of all maximal
chains by declaring how to enumerate the maximal chains corresponding
to the same permutation. 

For all of $\Bip(X)$, the proposed enumeration of all permutations of
$X$ will be obtained by the classical Johnson--Trotter algorithm
\cite{JohnAZ,Trotter}. For proper intervals in $\Bip(X)$, we will need
to consider a variant adapted to enumerate only specific subgroups of
the full permutation group (namely Young subgroups, although this term
will be of no importance in the sequel; the interested reader may consult 
\cite[Sec.~7.18]{StanBI} for more information). We recall the
Johnson--Trotter algorithm next, and subsequently describe its variant.
In Theorem~\ref{T_Trotter} we prove a property of the Johnson--Trotter
algorithm and of its variant which will be crucial in proving the 
key lemma, Lemma~\ref{L_skip}, and of its adaptation to the results of
Section~\ref{sec:int}, showing that the enumerations of maximal chains
that we construct grows by creating skipped intervals.  
The adaptation of Lemma~\ref{L_skip} to the case of intervals 
in Section~\ref{sec:int} is made possible by the introduction of the
{\em J--T decomposition} of the order complex of an interval in
Definition~\ref{D_Trotterdec}, and by Theorem~\ref{T_Trotterdec}, 
discussing the properties of this decomposition. 

\medskip
The original version of the Johnson--Trotter
algorithm~\cite{JohnAZ,Trotter} is used to
enumerate all permutations of $\{1,2,\ldots,n\}$ in such a way that each
permutation differs from the preceding one by a transposition of
adjacent elements. It may be described recursively as follows.
\begin{enumerate}
\item The Johnson--Trotter enumeration of all permutations of $\{1\}$ is $(\{1\})$.
\item Assume we are given the Johnson--Trotter enumeration of all permutations of
   $\{1,2,\ldots,n-1\}$. If the permutation
   $(\{\sigma_1\},\ldots,\{\sigma_{n-1}\})$
   is an odd numbered item in this enumeration, then we replace it with
   the list 
$$
(\{\sigma_1\},\ldots,\{\sigma_{n-1}\},\{n\}),\quad 
   (\{\sigma_1\}, \ldots,\{n\},\{\sigma_{n-1}\}),\quad  \ldots,\quad 
   (\{n\},\{\sigma_1\},\ldots,\{\sigma_{n-1}\}).
$$
Otherwise we replace
   it with the list 
$$(\{n\},\{\sigma_1\},\ldots,\{\sigma_{n-1}\}),\quad 
   (\{\sigma_1\},\{n\},\ldots,\{\sigma_{n-1}\}),\quad  \ldots,\quad 
   (\{\sigma_1\},\ldots, \{\sigma_{n-1}\}, \{n\}).
$$
\end{enumerate}

For example, the Johnson--Trotter enumeration of all permutations of 
$\{1,2,3\}$ is
\begin{gather*}
(\{1\},\{2\},\{3\}),\quad  (\{1\},\{3\},\{2\}),\quad
(\{3\},\{1\},\{2\}),\\
 (\{3\},\{2\},\{1\}),\quad  (\{2\},\{3\},\{1\}),\quad
(\{2\},\{1\},\{3\}). 
\end{gather*}

Before we are able to describe the announced variant, we need to first
review some facts about ordered partitions and their relation to
bipartitions. 

\begin{definition}
We say that an ordered partition $\pi$ of $X$ {\em refines}
the ordered partition $(C_1,\ldots,C_k)$ if each block $C_i$ is the
union of consecutive blocks of $\pi$.
\end{definition}

\begin{lemma}
For an ordered partition $\pi$ of $X$, a bipartitional relation $U(B_1^{\eps_1},
B_2^{\eps_2},\ldots,\break B_k^{\eps_k})\subseteq X\times X$ is $\pi$-compatible
if and only if $\pi$ refines the ordered partition $(B_1,
B_2,\ldots,\break B_k)$.
\end{lemma}

Refinement defines a partial order on the ordered partitions of $X$. The poset
thus obtained is isomorphic to the subposet of $\Bip(X)$ formed by all
bipartitions having only underlined blocks. This isomorphism is made precise in
the following definition.
\begin{definition}
Let $(C_1, C_2,\ldots, C_k)$ be an ordered partition of $X$. We define
the {\em underlined representation} 
of $(C_1, C_2,\ldots, C_k)$ 
in $\Bip(X)$ as the bipartitional relation $U(C_1^1,
C_2^1,\ldots, C_k^1)$, and we denote it by
$\underline{U}(C_1, C_2,\ldots, C_k)$. 
\end{definition}

\begin{lemma}
\label{C_uop}
Let $\pi$ and $\rho$ be ordered partitions of $X$. Then $\pi$ refines $\rho$
if and only if $\underline{U}(\pi)\leq \underline{U}(\rho)$ in $\Bip(X)$.
\end{lemma}

If $U=U(B_1^{\eps_1}, B_2^{\eps_2},\ldots,
B_k^{\eps_k})$ and $V=U(C_1^{\eta_1}, C_2^{\eta_2},\ldots,
C_l^{\eta_l})$, then we will be interested in finding all permutations
refining the ordered partition $(B_1, B_2,\ldots, B_k)$, as well as
the ordered partition $(C_1, C_2,\ldots, C_l)$.

\begin{corollary} \label{cor:refine}
A permutation $\pi$ of $X$ refines the ordered partitions $(B_1,
B_2,\ldots, B_k)$ and $(C_1, C_2,\ldots, C_l)$ if and only if
$$\underline{U}(\pi)\leq \underline{U}(B_1,
B_2,\ldots, B_k)\wedge \underline{U}(C_1, C_2,\ldots, C_l).$$
\end{corollary}
Note that $\underline{U}(B_1,
B_2,\ldots, B_k)\wedge \underline{U}(C_1, C_2,\ldots, C_l)$
is taken in $\Bip(X)$. Hence, the resulting bipartitional relation may also
have nonunderlined blocks in its ordered bipartition
representation. By Proposition~\ref{P_cont}, such a bipartitional
relation cannot contain a bipartition having only underlined blocks,
and in that case there is no permutation refining both $(B_1,
B_2,\ldots, B_k)$ and $(C_1,C_2,\ldots, C_l)$. If, however,
  $\underline{U}(B_1,
B_2,\ldots, B_k)\wedge \underline{U}(C_1, C_2,\ldots, C_l)$ has an
ordered bipartition representation $(D_1^1,\ldots,D_m^1)$ consisting of
underlined blocks only, then $\pi$ refines both $(B_1,
B_2,\ldots, B_k)$ and $(C_1,C_2,\ldots, C_l)$ if and only if it refines
$(D_1,\ldots,D_m)$. Therefore enumerating all permutations associated to
some maximal chain containing $U$ and $V$ is equivalent to enumerating
all permutations refining a given ordered partition. 

Now we adapt the Johnson--Trotter algorithm to list all permutations 
refining a given ordered partition $(C_1,\ldots,C_k)$ of
$\{1,2,\ldots,n\}$ such that each permutation differs from
the preceding one by a transposition of adjacent blocks, as follows.
\begin{enumerate}
\item For $n=1$ we may only have $k=1$, $C_1=\{1\}$, and we list the
   permutation $(\{1\})$.
\item Assume we already know how to build the Johnson--Trotter enumeration
of all permutations refining any given ordered partition of
$\{1,2,\dots,n-1\}$. Let $(C_1,\ldots,C_k)$ be an ordered partition
of $\{1,2,\dots,n\}$. If $\{n\}=C_m$ is a block by itself, then
$(C_1,\ldots,C_{m-1},C_{m+1},\ldots,C_k)$ is an ordered partition of
$\{1,2,\dots,n-1\}$. List all permutations $\pi$ of $\{1,2,\dots,n-1\}$ refining
$(C_1,\ldots,C_{m-1},C_{m+1},\ldots,\break C_n)$, and insert $\{n\}$ between
the last block of $\pi$ contained in $C_{m-1}$ and the first block of $\pi$
contained in $C_{m+1}$. (At most one of these blocks may be missing if
$m=1$ or $m=k$.) We obtain an appropriate enumeration. 

From now on we
may assume that the block $C_m$ containing $n$ contains at least one more
element.  Introducing $C_j'=C_j\setminus \{n\}$ for $j=1,\ldots,k$,
$(C_1',\ldots,C_k')$ is an ordered partition of
$\{1,2,\ldots,n-1\}$, and we consider the Johnson--Trotter enumeration of all
permutations refining this ordered partition. Let
$$\pi=(\{\pi_1\},\ldots,\{\pi_r\},\ldots,\{\pi_s\},\ldots,\{\pi_n\})$$
be a permutation in this enumeration, where
$C_m'=\{\pi_r,\ldots,\pi_s\}$. (Note that $r$ and $s$ are the same for all
permutations in the enumeration.) If $\pi$ is an odd numbered item in
this enumeration, then we replace it with the list
\begin{multline*}
\kern1cm
(\{\pi_1\},\ldots,\{\pi_r\},\ldots,\{\pi_s\},\{n\},\ldots,\{\pi_n\}),\\
(\{\pi_1\},\ldots,\{\pi_r\},\ldots,\{n\},\{\pi_s\},\ldots,\{\pi_n\}),\\
\ldots, \quad (\{\pi_1\},\ldots,\{n\},\{\pi_r\},\ldots,\{\pi_s\},\ldots,
\{\pi_n\}),
\end{multline*}
otherwise we replace it with the list
\begin{multline*}
\kern1cm
(\{\pi_1\},\ldots,\{n\},\{\pi_r\},\ldots,\{\pi_s\},\ldots,
\{\pi_n\}),\\
(\{\pi_1\},\ldots,\{\pi_r\},\{n\},\ldots,\{\pi_s\},\ldots,\{\pi_n\}),\\
\dots,\quad 
(\{\pi_1\},\ldots,\{\pi_r\},\ldots,\{\pi_s\},\{n\},\ldots,
\{\pi_n\}).
\end{multline*}
\end{enumerate}
For example, the Johnson--Trotter enumeration of all permutations refining
$(\{1,3\},\break \{2,4\})$ is built recursively as follows.
\begin{enumerate}
\item The list of all permutations refining $(\{1\})$ is $(\{1\})$.
\item The list of all permutations refining $(\{1\},\{2\})$ is $(\{1\},\{2\})$.
\item The list of all permutations refining $(\{1,3\},\{2\})$ is
$$(\{1\},\{3\},\{2\}),\quad  (\{3\},\{1\}, \{2\}).$$
\item The list of all permutations refining $(\{1,3\},\{2,4\})$ is
\begin{gather*}
(\{1\},\{3\},\{2\},\{4\}),\quad (\{1\},\{3\},\{4\}, \{2\}),\\
(\{3\},\{1\},\{4\},\{2\}),\quad  (\{3\},\{1\},\{2\},\{4\}).
\end{gather*}
\end{enumerate}

Recall (see \cite[p.~324]{BjoeAA})
that the {\it shelling} of a simplicial complex $\triangle$ is an
enumeration $F_1,\ldots, F_m$ of its facets such that each facet has the
same dimension as $\triangle$ and, for each $i>1$, any face $\tau\subset
F_i$ that is contained in some preceding $F_k$, is also
contained in a preceding $F_j$ whose intersection $F_j\cap F_i$ with
$F_i$ has codimension one (that is,
making $F_j\cap F_i$ ``as large as possible'').
Our next goal is to decompose each order complex 
$\triangle([U,V]\setminus \{U,V\})$ in a manner resembling a shelling.
In the decomposition we are going to describe, the role of the facets
in a shelling will be played by subcomplexes of the form
\begin{equation} \label{eq:triangle} 
\triangle([U,V]_\pi\setminus \{U,V\}),
\end{equation}
where $\pi$ is a permutation such that $U$ and $V$ are
$\pi$-compatible, and $[U,V]_\pi$ stands for the subposet of $[U,V]$
consisting of all $\pi$-compatible bipartitional relations.
As a consequence, we need to understand intersections of subcomplexes
of the form \eqref{eq:triangle}. For this, it is necessary to know
when a bipartitional relation is simultaneously compatible with two
different ordered partitions. In this regard, Lemma~\ref{C_uop} yields
the following immediate characterization.

\begin{corollary}
\label{C_simulc}
A bipartitional relation $U$ is simultaneously compatible with the
ordered partitions $(B_1,\ldots, B_k)$ and $(C_1,\ldots,C_l)$ if and
only if $U$ is compatible with the ordered partition $(D_1,\ldots,D_m)$
given by
$$
\underline U(D_1,\ldots,D_m)=\underline U(B_1,\ldots, B_k) \vee
\underline U(C_1,\ldots,C_l).
$$
\end{corollary}
Obviously $(D_1,\ldots,D_m)$ is the ``finest common coarsening'' of the
ordered partitions $(B_1,\ldots, B_k)$ and $(C_1,\ldots,C_l)$.
In particular, in the case where $X=\{1,2,\dots,n\}$ and the permutations
$\sigma=(\{\sigma_1\},\ldots, \{\sigma_n\})$  and
$\pi=(\{\pi_1\},\ldots, \{\pi_n\})$ differ in a
transposition of adjacent blocks, say,
$\pi=(\{\sigma_1\},\ldots,\{\sigma_{i+1}\},\{\sigma_i\},\ldots
\{\sigma_n\})$,
then
$$\underline U(\sigma)\vee \underline U(\pi)=(\{\underline{\sigma_1}\},
\ldots,
\{\underline{\sigma_{i}},\underline{\sigma_{i+1}}\},
\ldots,\{\underline{\sigma_n}\}).$$
Thus $\underline U(\sigma)\vee \underline U(\pi)$ covers $\underline
U(\sigma)$ and $\underline U(\pi)$ in $\Bip(\{1,2,\dots,n\})$, and the
   intersection\break $\Bip_\sigma(\{1,2,\dots,n\})\cap
     \Bip_\pi(\{1,2,\dots,n\})$ has the largest possible rank that a
       proper intersection of two lattices of the form
       $\Bip_\sigma(\{1,2,\dots,n\})$ may have. As a consequence, for
every interval $[U,V]$ such that $U$ and $V$ are simultaneously
$\sigma$- and $\pi$-compatible, the
subcomplexes $\triangle([U,V]_\sigma\setminus
\{U,V\})$ and $\triangle([U,V]_\pi\setminus \{U,V\})$ are either equal
or their intersection has codimension $3$ in both subcomplexes (recall
the formula \eqref{E_rank} for the rank function and the fact that,
in the latter case,
the above intersection arises by identifying $\sigma_i$ and
$\sigma_{i+1}$). This makes such
       intersections analogous to 
       codimension $1$ faces in a shelling (recall again the
definition \cite[p.~324]{BjoeAA} of a shelling).
 
The following technical result is the key result for establishing the
above indicated shelling-like property of the announced J--T
decomposition (to be defined in Definition~\ref{D_Trotterdec}) 
in Theorem~\ref{T_Trotterdec}.(iii).

\begin{theorem}
\label{T_Trotter}
Let $(C_1,\ldots,C_k)$ be an ordered partition of the set
$\{1,2,\dots,n\}$ and consider the Johnson--Trotter enumeration of all
permutations refining this ordered partition. If $\tau$ precedes
$\sigma$ in this enumeration, then there is a permutation $\pi$ preceding
$\sigma$ in this enumeration, 
which differs from $\sigma$ only in a transposition of adjacent
blocks, and which satisfies
$$
\underline U(\tau)\vee \underline U(\sigma)\geq \underline U (\pi)\vee
\underline U (\sigma) \gtrdot \underline U (\sigma).
$$
Here, $V_1\gtrdot V_2$ means that $V_1$ covers $V_2$.
\end{theorem}

\begin{example}
The permutation $\tau=(\{4\},\{1\},\{2\},\{3\})$ precedes 
$\sigma=(\{4\},\{2\},\{3\},\{1\})$ in the Johnson--Trotter enumeration
of all permutations of $\{1,2,3,4\}$, and we have 
$$
\underline U(\tau)\vee \underline U(\sigma)=
(\{\underline 4\},\{\underline 1, \underline 2, \underline 3\}).
$$
The permutation $\pi=(\{4\},\{3\},\{2\},\{1\})$ also precedes $\sigma$ in the
Johnson--Trotter enumeration, and $\pi$ differs from $\sigma$ only in the
transposition of the adjacent blocks $\{2\}$ and $\{3\}$. Thus we have 
$$
\underline U (\pi)\vee \underline U (\sigma)=
(\{\underline 4\},\{\underline 2, \underline 3\}, \{\underline 1\}),
$$
implying
$$
(\{\underline 4\},\{\underline 1, \underline 2, \underline 3\})
\geq (\{\underline 4\},\{\underline 2, \underline 3\}, \{\underline 1\})
\gtrdot (\{\underline 4\},\{\underline
2\},\{\underline 3\},\{\underline 1\}).
$$
We should perhaps also point out that, here, 
$\pi$ {\it does not immediately} 
precede $\sigma$ in the Johnson--Trotter enumeration.
\end{example}

\begin{proof}[Proof of Theorem~\ref{T_Trotter}]
We prove the statement by induction on $n$. There is nothing to prove
for $n=1$. Assume that the statement holds up to $n-1$.
Consider a $\tau$
preceding a $\sigma$ in the Johnson--Trotter enumeration of all
permutations refining the ordered partition $(C_1,\ldots,C_k)$ of
$\{1,2,\dots,n\}$. Let $\sigma\setminus n$, respectively $\tau\setminus
n$, denote the permutations obtained from $\sigma$, respectively $\tau$, by
deleting the block $\{n\}$. Let $(C_1',\ldots,C_l')$ be the ordered
partition considered right before inserting $n$ in the Johnson--Trotter
enumeration associated to $(C_1,\ldots,C_k)$. In other words, we have
$l=k-1$ and
$(C_1',\ldots,C_{k-1}')=(C_1,\ldots,C_{m-1},C_{m+1},\ldots,C_k)$ if
$\{n\}=C_m$ is a block by itself
(for some $m$), and we have $l=k$ and
$C_j'=C_j\setminus \{n\}$ for all $j$ otherwise. Note that $\sigma\setminus n$
and $\tau\setminus n$ belong to the set of all permutations of
$\{1,2,\dots,n-1\}$ refining $(C_1',\ldots,C_l')$.

\medskip {\sc Case 1.}  $\sigma\setminus n\neq \tau\setminus n$.
Then $\tau\setminus n$ precedes $\sigma\setminus n$ in the Johnson--Trotter
enumeration of all permutations refining $(C_1',\ldots,C_l')$. By our
induction hypothesis, there is a permutation $\pi'$ preceding
$\sigma\setminus n$ in this enumeration such that we have
\begin{equation}
\label{E_indhyp}
\underline U(\tau\setminus n)\vee \underline U(\sigma\setminus n)\geq
\underline U (\pi')\vee
\underline U (\sigma\setminus n) \gtrdot \underline U (\sigma\setminus n).
\end{equation}
Here, for some $i$, we have
$$
\underline U (\pi')\vee\underline U (\sigma\setminus n)
=(\{\underline{\sigma_1}\},\ldots,
\{\underline{\sigma_i}, \underline{\sigma_{i+1}}\},\ldots, \{\underline{\sigma_{n-1}}\}),
$$
and we may assume
$$
\sigma\setminus n
=(\{\sigma_1\},\ldots,\{\sigma_i\},\{\sigma_{i+1}\},\ldots, \{\sigma_{n-1}\})
$$
and 
$$
\pi'=(\{\sigma_1\},\ldots,
\{\sigma_{i+1}\},\{\sigma_{i}\},\ldots,\{\sigma_{n-1}\}).
$$

\smallskip
{\it Subcase 1a.} $\sigma=
(\{\sigma_1\},\ldots,\{\sigma_j\},\{n\},\{\sigma_{j+1}\},\ldots,
\{\sigma_{n-1}\})$ holds for some $j\neq i$. Then we set
$$
\pi:=(\{\pi_1'\},\ldots,\{\pi'_j\},\{n\},\{\pi'_{j+1}\},\ldots,
\{\pi'_{n-1}\}),
$$
where $\pi'=(\{\pi_1'\},\ldots,\{\pi_{n-1}'\})$. In other words, we
insert $\{n\}$ at the same place into $\pi'$ as the place where it needs
to be inserted into $\sigma\setminus n$ to obtain $\sigma$. By the recursive
structure of the Johnson--Trotter enumeration, $\pi$ precedes $\sigma$ in the
enumeration of all permutations refining $(C_1,\ldots,C_k)$. Moreover,
$\pi$ differs from $\sigma$ only in transposing the adjacent blocks
$\{\sigma_i\}$ and $\{\sigma_{i+1}\}$. Thus, 
\begin{equation} \label{eq:pisi} 
\underline U (\pi)\vee \underline U (\sigma)=
(\{\underline{\sigma_1}\},\ldots,
\{\underline{\sigma_i},\underline{\sigma_{i+1}}\},\ldots,
\{\underline{\sigma_{n}}\}),
\end{equation}
and it covers $\underline U
(\sigma)$ in $\Bip(\{1,2,\dots,n\})$. Finally, by (\ref{E_indhyp}),
the ordered bipartition representation of $\underline U(\tau\setminus
n)\vee \underline U(\sigma\setminus n)$ contains $\sigma_i$ and
$\sigma_{i+1}$ in the same underlined block. Therefore $\sigma_i$ and
$\sigma_{i+1}$ are also in the
same underlined block in the ordered bipartition representing $\underline
U(\tau)\vee \underline U(\sigma)$- Together with \eqref{eq:pisi}, this
implies
$\underline U(\tau)\vee \underline U(\sigma)\geq \underline U(\pi)\vee
\underline U(\sigma)$.

\smallskip
{\it Subcase 1b.} $\sigma=
(\{\sigma_1\},\ldots,\{\sigma_i\},\{n\},\{\sigma_{i+1}\},\ldots,
\{\sigma_{n-1}\})$. In the same way as at the end of the previous subcase,
the fact that $\sigma_i$ and $\sigma_{i+1}$ belong to the same
underlined block of  $\underline U(\tau\setminus
n)\vee \underline U(\sigma\setminus n)$ implies that they also belong to
the same underlined block of $\underline U(\tau)\vee \underline
U(\sigma)$. Now, $(\sigma_i,n)\in \underline U(\sigma)$ and
$(n,\sigma_{i+1})\in \underline U(\sigma)$ imply that $n$ also belongs
to the same underlined block of 
$\underline U(\tau)\vee\underline U(\sigma)$. 
Thus $\underline U(\tau)\vee \underline
U(\sigma)$ contains the bipartitional relation $V$ represented by the
  ordered bipartition
$$(\{\underline{\sigma_1}\},\ldots,\{\underline{\sigma_{i-1}}\},
  \{\underline{\sigma_i},\underline{n},\underline{\sigma_{i+1}}\},\{\underline{\sigma_{i+2}}\}, \ldots, 
\{\underline{\sigma_{n-1}}\}).$$ 
As a consequence, by
  Lemma~\ref{C_uop}, the block $C_m$ containing $n$ also contains
  $\sigma_i$ and $\sigma_{i+1}$. Therefore the permutations
$$
\rho'=(\{\sigma_1\},\ldots,\{n\},\{\sigma_i\},\{\sigma_{i+1}\},\ldots,
\{\sigma_{n-1}\})$$
and 
$$
\rho''=(\{\sigma_1\},\ldots,\{\sigma_i\},\{\sigma_{i+1}\},\{n\},\ldots,
\{\sigma_{n-1}\})
$$
both refine $(C_1,\ldots,C_k)$. By the structure of the Johnson--Trotter
enumeration, one of them precedes $\sigma$. This one may be chosen as
$\pi$. It follows that
$$
\underline U(\tau)\vee \underline U(\sigma)\geq V
\gneqq\underline U(\pi)\vee \underline U(\sigma)
\gtrdot \underline U(\sigma).
$$

\medskip {\sc Case 2.}  $\sigma\setminus n= \tau\setminus n$.
Without loss of generality, we may assume that $\sigma\setminus n=
\tau\setminus n$ is an even numbered item in the Johnson--Trotter enumeration of
all permutations refining $(C_1',\ldots,C_l')$. Since $\tau$ precedes
$\sigma$ in the Johnson--Trotter enumeration of all permutations refining
$(C_1,\ldots,C_k)$, we must have
$$
\tau=(\{\sigma_1\},\ldots,\{\sigma_i\},\{n\},\{\sigma_{i+1}\},\ldots,
\{\sigma_{n-1}\})$$
and 
$$
\sigma=(\{\sigma_1\},\ldots,\{\sigma_j\},\{n\},\{\sigma_{j+1}\},\ldots,
\{\sigma_{n-1}\})
$$
for some $i<j\leq n-1$. It is easy to see that
$$\underline U(\tau)\vee \underline
U(\sigma)
=U(\{\underline{\sigma_1}\}, \ldots,
\{\underline{\sigma_{i}}\},\{\underline{\sigma_{i+1}},\ldots  
\underline{\sigma_{j}},\underline{n}\},\{\underline {\sigma_{j+1}}\}, \ldots,   
\{\underline{\sigma_{n-1}}\}).$$ 
As a consequence, by
  Lemma~\ref{C_uop}, the block $C_m$ containing $n$ also contains
$\sigma_{i+1},\ldots \sigma_{j}$. Therefore the permutation
$$
\pi=(\{\sigma_1\},\ldots,\{\sigma_{j-1}\},\{n\},\{\sigma_{j}\},\ldots,
\{\sigma_{n-1}\})
$$
also refines $(C_1,\ldots,C_k)$ and it precedes $\sigma$ in the Johnson--Trotter
enumeration. We have
$$\underline U(\pi)\vee \underline U(\sigma)
=(\{\underline{\sigma_1}\},\ldots,\{\underline{\sigma_{j-1}}\},
\{\underline{n},\underline{\sigma_{j}}\},
\{\underline{\sigma_{j+1}}\}\ldots,
\{\underline{\sigma_{n-1}}\}).$$
In particular, $\underline U(\pi)\vee \underline U(\sigma)$
covers $\underline U(\sigma)$, and
$\underline U(\pi)\vee \underline U(\sigma)$ is contained in
$\underline U(\tau)\vee \underline U(\sigma)$.
\end{proof}

Using Theorem~\ref{T_Trotter}, we now show that, for any interval
$[U,V]\subseteq \Bip(X)$, the order complex $\triangle([U,V]\setminus
\{U,V\})$ has a ``shelling-like" decomposition. Namely, we may write
\begin{equation} \label{eq:preshell} 
\triangle([U,V]\setminus \{U,V\})=\bigcup_{\pi}
\triangle([U,V]_\pi\setminus \{U,V\}),
\end{equation}
where the union is taken over all permutations $\pi$ such that
both $U$ and $V$ are $\pi$-compatible.
(The notation $[U,V]_\pi$ was defined just after \eqref{eq:triangle}.)
We may enumerate these permutations using the Johnson--Trotter
enumeration.
To see the similarity with a shelling, the reader should imagine that
the role of facets in a shelling is played in the decomposition
\eqref{eq:preshell} by the subcomplexes 
$\triangle([U,V]_\pi\setminus \{U,V\})$, all of which have the same
dimension as $\triangle([U,V]\setminus \{U,V\})$ by
Corollary~\ref{cor:samerank}. Moreover, 
by Theorem~\ref{T_dist}, each poset $[U,V]_{\pi}$ is a
distributive lattice and, by a result due to Provan \cite{ProvAA}
(cf.\ also \cite[Cor.~2.2]{Billera-Hsiao-Provan}),
the order complex $\triangle([U,V]_\pi\setminus \{U,V\})$ is either
the order complex of a Boolean lattice (and thus isomorphic to the
boundary complex of a simplex) or it is a polyhedral ball. In the case
where $U=\emptyset$ and $V=X\times X$, i.e., when $[U,V]=\Bip(X)$, the
sublattice $[U,V]_{\pi}$ is never a Boolean lattice, thus we decompose
the order complex as a union of balls. 
Note next that for proper intervals $[U,V]\subset \Bip(X)$ it may happen
that $\triangle([U,V]_\pi\setminus \{U,V\})=\triangle([U,V]_\sigma\setminus
\{U,V\})$ holds for some $\pi\neq \sigma$. For example, in the case
where the ordered bipartition $(B_1^{\eps_1}, B_2^{\eps_2},\ldots,
B_k^{\eps_k})$ representing $U$ and the ordered bipartition
$(C_1^{\eta_1}, C_2^{\eta_2},\ldots,C_l^{\eta_l})$ representing $V$
satisfy $\eps_1=\eta_1=0$ and $B_1=C_1=\{1,2\}$ then a permutation
$\pi=(\{\pi_1\},\ldots, \{\pi_n\})$ refining both $(B_1,
B_2,\ldots, B_k)$ and $(C_1, C_2,\ldots,C_l)$ must satisfy
$\{\pi_1,\pi_2\}=\{1,2\}$, but it does not matter whether $\pi_1=1$ and
$\pi_2=2$ or $\pi_1=2$ and $\pi_2=1$. By Proposition~\ref{P_c}, $\pi$
arises as the only ordered partition that is compatible with all
elements of some maximal chain $c$ containing $U$ and $V$, 
but the choice of the values of $\pi_1$ and $\pi_2$ is related to the
part of $c$ that is outside the interval $[U,V]$. We may overcome this
difficulty by keeping only the first copy of each
$\triangle([U,V]_\pi\setminus \{U,V\})$.

\begin{definition}
\label{D_Trotterdec}
Let $[U,V]\subseteq \Bip(X)$ be an interval, where $U=(B_1^{\eps_1},
B_2^{\eps_2},\ldots,  B_k^{\eps_k})$, $V=(C_1^{\eta_1},
C_2^{\eta_2},\ldots,C_l^{\eta_l})$, and let the ordered partition
$(D_1,\ldots,D_m)$ be given by 
$$\underline U(D_1,\ldots,D_m)=\underline
U(B_1, B_2,\ldots,B_k)\wedge \underline U(C_1, C_2,\ldots,C_l).$$
We define the {\em J--T decomposition } of $\triangle([U,V]\setminus \{U,V\})$ as follows:
\begin{enumerate}
\item We list the order complexes $\triangle([U,V]_\pi\setminus
\{U,V\})$ in
the order of the Johnson--Trotter enumeration of permutations $\pi$ refining
$(D_1,\ldots,D_m)$. (By Cor\-ol\-lary~\ref{cor:refine}, these are the
permutations $\pi$ such that $U$ and $V$ are both $\pi$-compatible.)
\item If the same simplicial complex occurs several times in the above
   enumeration, we keep only its first occurrence and remove all
   other occurrences.
\item The remaining list $\triangle_1, \ldots, \triangle_N$ is the
   J--T decomposition of $\triangle([U,V]\setminus \{U,V\})$.
\end{enumerate}
\end{definition}
\begin{theorem}
\label{T_Trotterdec}
The J--T decomposition $\triangle_1, \ldots, \triangle_N$ of
$\triangle([U,V]\setminus \{U,V\})$ has the following properties:
\begin{itemize}
\item[(i)] each $\triangle_i$ has the same dimension as
   $\triangle([U,V]\setminus \{U,V\})$;
\item[(ii)] each $\triangle_i$  is either isomorphic to the boundary
   complex of a simplex or it is a polyhedral ball;
\item[(iii)] for $i>1$, any face contained in
   $\triangle_i\cap\left(\bigcup_{j<i} \triangle_j\right)$ is also contained
   in $\triangle_k$ for some $k<i$ such that there are permutations
   $\sigma$ and $\pi$ that differ only in a transposition
of adjacent blocks, with
   $\triangle_i=\triangle([U,V]_\sigma\setminus \{U,V\})$ and
   $\triangle_k=\triangle([U,V]_\pi\setminus \{U,V\})$.
\end{itemize}
\end{theorem}
\begin{proof}
We only need to show (iii) since, as mentioned above, item~(i) follows from
Corollary~\ref{cor:samerank}, and item~(ii) follows
from Provan's result~\cite{ProvAA,Billera-Hsiao-Provan}. Assume
$\triangle_i=\triangle([U,V]_\sigma\setminus \{U,V\})$  and consider a
face $\gamma$ that is also contained in $\triangle_j=\triangle([U,V]_\tau\setminus
\{U,V\})$ for some $j<i$. This means that the elements of
$\gamma$ are $\sigma$-compatible and $\tau$-compatible bipartitional
relations. By Corollary~\ref{C_simulc}, these bipartitional relations are also
compatible with the ordered bipartition $\rho$ given by
$\underline U(\rho)=\underline U(\tau)\vee\underline U(\sigma)$.
By Theorem~\ref{T_Trotter}, there is a permutation $\pi$ preceding
$\sigma$ in the Johnson--Trotter enumeration of all permutations
$\rho$ with the property
that both $U$ and $V$ are $\rho$-compatible, such that
$\underline U(\tau)\vee \underline U(\sigma)\geq \underline U (\pi)\vee
\underline U (\sigma) \gtrdot \underline U (\sigma)$.
Thus the
elements of the face $\gamma$ are also $\pi$-compatible, and $\gamma$ is
contained in 
$$\triangle([U,V]_\sigma\setminus \{U,V\})\cap
\triangle([U,V]_\pi\setminus \{U,V\}).$$ 
Here $\triangle([U,V]_\pi\setminus
\{U,V\})=\triangle_k$ for some $k<i$ in the J--T decomposition,
in particular, $\triangle_k\neq \triangle_i$.
\end{proof}

\section{The topology of the order complex of
$\Bip(X)\setminus\{\emptyset, X\times X\}$}
\label{s_toc}

In this section we determine the homotopy type of the order complex of
$\Bip(X)\setminus \{\emptyset, X\times X\}$. To achieve this goal, we
construct a listing of all maximal chains contained in $\Bip(X)$ and then
use the results of Babson and Hersh~\cite{Babson-Hersh} as described in
Section~\ref{sec:Morse} to see that there is exactly one critical
cell in this order complex with respect to the induced discrete Morse
matching.

We now describe the announced listing of all maximal chains of
$\Bip(X)$.
Without loss of generality, we may
assume that $X=\{1,2,\ldots,n\}$. The construction involves the
following three steps.

\smallskip
{\sc Step 1.}
By Proposition~\ref{P_c}, for each maximal
chain $c$ in $\Bip(X)$, there is a unique permutation $\sigma$ such
that all elements of $c$ are $\sigma$-compatible. Let us list the
permutations of $X$ using the Johnson--Trotter enumeration and associate to each
permutation $\sigma$ the set of all $\sigma$-compatible maximal
chains, or, equivalently, of {\it all\/} maximal chains in $\Bip_\sigma(X)$. 

\smallskip
{\sc Step 2.}
By Theorem~\ref{T_dist}, for a fixed $\sigma$, the lattice
$\Bip_\sigma(X)$ is distributive. By \cite[Th.~4.5]{BjoeAA},
it has an $EL$-labelling using its join-irreducible elements. In this
$EL$-labelling, an edge $UV$, where $U$ and $V$ are elements of
$\Bip_\sigma(X)$ 
such that $U$ is covered by $V$, is labelled by the unique
join-irreducible element $W\in \Bip_\sigma(X)$ such that $W\subseteq
V$ but $W\not\subseteq U$. We use this $EL$-labelling to 
order the maximal chains of $\Bip_\sigma(X)$ in the following
way. The join-irreducible elements of $\Bip_\sigma(X)$ for
$\sigma=(\{1\},\{2\},\ldots,\{n\})$ are given in
Proposition~\ref{P_jirred}. By permuting the elements of $X$, it is easy to
see that the join-irreducible elements of $\Bip_\sigma(X)$ are the following:
\begin{itemize}
\item[(i)]
   $E(\sigma,i):=(\{\sigma_1,\ldots,\sigma_{i-1}\},\{\sigma_{i},\ldots,\sigma_n\})$
   for $i\in\{2,\ldots,n\}$,
\item[(ii)]
   $F(\sigma,i):=(\{\sigma_1,\ldots,\sigma_{i-1}\},\{\underline{\sigma_i}\},\{\sigma_{i+1},\ldots,\sigma_n\})$
   for $i\in\{1,2,\dots,n\}$,
\item[(iii)]
   $G(\sigma,i):=(\{\sigma_1,\ldots,\sigma_{i-1}\},\{\underline{\sigma_i},\underline{\sigma_{i+1}}\},\{\sigma_{i+2},\ldots,\sigma_n\})$
   for $i\in\{1,2,\dots,n-1\}$.
\end{itemize}
For all what follows in this section, we fix the linear extension
\begin{gather}
\notag
E(\sigma,2)\prec F(\sigma,1)\prec E(\sigma,3)\prec F(\sigma,2)\prec G(\sigma,1)\prec \cdots\\
\notag
\prec E(\sigma,k+2)\prec F(\sigma,k+1)\prec G(\sigma,k)\prec \cdots\\
E(\sigma,n)\prec F(\sigma,n-1)\prec G(\sigma,n-2)\prec F(\sigma,n)\prec G(\sigma,n-1).
\label{E_linext}
\end{gather}
(Here, we use the symbol $\prec$ to distinguish the linear extension
in \eqref{E_linext} from the order relation in $\Bip_\sigma(X)$.) 
Now, as announced above, we associate to each maximal chain 
$c: \emptyset= U_0\lessdot \cdots \lessdot U_{3n-2}=X\times X$
contained in $\Bip_\sigma(X)$
the word $z_1z_2\cdots z_{3n-2}$, where the
letter $z_i$ is the unique join-irreducible element contained in $U_i$ but not
contained in $U_{i-1}$.
We list the maximal chains in $\Bip_\sigma(X)$ according to the
lexicographic order of their associated words.

\smallskip
{\sc Step 3.}
Given a $\sigma$-compatible maximal chain $c$ and a
$\sigma'$-compatible maximal chain $c'$, the chain $c$ precedes $c'$ if
and only if either $\sigma$ precedes $\sigma'$ in the Johnson--Trotter
enumeration, or if $\sigma=\sigma'$ and $c$ precedes $c'$ in the ordering
of the maximal chains in $\Bip_\sigma(X)$ described in Step~2.

\medskip
The list of  chains we thus obtain is in general {\em not} a poset
lexicographic order as defined in \cite{Babson-Hersh}, 
as may be seen in the following example.

\begin{example}\label{e_nonplo}
Consider the cover relations $U(\{1,2,3,4\})\lessdot U(\{1,2\},\{3,4\})$\break
and $U(\{1,2,3,4\})\lessdot U(\{1,4\},\{2,3\})$ in $\Bip(\{1,2,3,4\})$.
Here $U(\{1,2\},\{3,4\})$ is\break 
$(\{1\},\{2\},\{3\},\{4\})$-compatible and
$(\{2\},\{1\},\{3\},\{4\})$-compatible, but not\break
$(\{1\},\{4\},\{2\},\{3\})$-compatible,
whereas $U(\{1,4\},\{2,3\})$ is $(\{1\},\{4\},\{2\},\{3\})$-compat\-ible
but not $(\{1\},\{2\},\{3\},\{4\})$-compatible nor
$(\{2\},\{1\},\{3\},\{4\})$-compatible . In the Johnson--Trotter enumeration of
all permutations of $\{1,2,3,4\}$, the permutations
$(\{1\},\{2\},\{3\},\{4\})$, $(\{1\},\{4\},\{2\},\{3\})$,
$(\{2\},\{1\},\{3\},\{4\})$ follow in this order. It is not true
that every maximal chain extending $U(\{1,2,3,4\})\lessdot
U(\{1,2\},\{3,4\})$ precedes every maximal chain extending
$U(\{1,2,3,4\})\lessdot U(\{1,4\},\{2,3\})$ since any\break
$(\{1\},\{4\},\{2\},\{3\})$-compatible maximal chain precedes any
$(\{2\},\{1\},\{3\},\{4\})$-com\-pat\-ible maximal chain. On the other hand,
it is also not true that every maximal chain extending 
$U(\{1,2,3,4\})\lessdot U(\{1,4\},\{2,3\})$
precedes every maximal chain extending $U(\{1,2,3,4\})\lessdot
U(\{1,2\},\{3,4\})$ since any
$(\{1\},\{2\},\{3\},\{4\})$-compatible maximal chain precedes any
$(\{1\},\{4\},\{2\},\{3\})$-com\-pat\-ible maximal chain.
\end{example}

However, our list of maximal chains still grows by skipped intervals, as
defined in Definition~\ref{D_ecsi}, which we prove in the lemma below.
In the proof of the lemma, and also later, we use the following
well-known property of the $EL$-labelling of distributive lattices
which we recalled in Step~2 above, and which involves the notion of a
{\it descent\/} in a word $z_1z_2\cdots z_{3n-2}$: an index $i$ for which 
$z_i\succ z_{i+1}$ is called a descent. When we list the maximal
chains of $\Bip_\sigma(X)$ in lexicographic order as described
in Step~2, then a subset $\{U_{i_1},\ldots,U_{i_k}\}$ of a maximal
chain $c$ with associated word $z_1z_2\cdots z_{3n-2}$ 
does not belong to any previously listed maximal chain if
and only if the set of their ranks contains all descents of 
$z_1z_2\cdots z_{3n-2}$.

\begin{lemma}
\label{L_skip}
Let $X=\{1,2,\dots,n\}$, and let 
$\sigma=(\{\sigma_1\},\{\sigma_2\},\dots,\{\sigma_n\})$ be a permutation of
$X$. Furthermore, 
let $c$ be a $\sigma$-compatible maximal chain of $\Bip(X)$, and assume
that the word $z_1z_2\cdots z_{3n-2}$ is associated to $c$ in
$\Bip_\sigma(X)$. 
As before, we list the maximal chains in $\Bip_\sigma(X)$ according
to the lexicographic order of their associated words, as described in
Step~2 at the beginning of this section.
Then a chain contained in $c$ is also contained in
an earlier listed maximal chain 
of $\Bip_\sigma(X)$ if and only if the set of its ranks is
disjoint from at least one of the following intervals:
\begin{itemize}
\item[(i)] all singletons $[i,i]=\{i\}$ such that $z_i\succ z_{i+1}$;
\item[(ii)] all intervals $[i,j]$ with $z_i=E(\sigma,q)$,
   $z_{j+1}=G(\sigma,q-1)$, for some $q$, such that the permutation
$\pi$ obtained from $\sigma$ by exchanging 
   the adjacent blocks $\{\sigma_{q-1}\}$ and $\{\sigma_{q}\}$ precedes
   $\sigma$ in the Johnson--Trotter enumeration.
\end{itemize}
\end{lemma}
\begin{proof}
Consider first the intersection of $c: \emptyset= U_0\lessdot \cdots \lessdot U_{3n-2}=X\times X$ with another $\sigma$-compatible
maximal chain that was listed earlier. As noted in the paragraph
preceding the lemma, a chain
belongs to such an intersection if and only if its set of ranks does not
contain the descent set $\{i\::\: z_i\succ z_{i+1}\}$ of the word
$z_1z_2\cdots z_{3n-2}$ or, equivalently, if it is disjoint from at least one
of the singletons listed in (i).

Consider next the intersection of $c$ with a $\tau$-compatible maximal
chain $c'$, 
where $\tau$ precedes $\sigma$ in the Johnson--Trotter enumeration. This
intersection is contained in $\Bip_\tau(X)\cap \Bip_\sigma(X)$
which, by part (iii) of Theorem~\ref{T_Trotterdec}, is contained in some
$\Bip_\pi(X)$, where $\pi$ is obtained from $\sigma$ by exchanging
   the adjacent blocks $\{\sigma_{q-1}\}$ and $\{\sigma_{q}\}$,
for some $q$, and
$\pi$ precedes
   $\sigma$ in the Johnson--Trotter enumeration. We may extend the
intersection $c\cap c'$ to a $\pi$-compatible maximal chain. Clearly, the
intersection $c\cap c'$ is 
then a chain in $\Bip_\pi(X)\cap \Bip_\sigma(X)$.
Consider now the word  $z_1z_2\cdots z_{3n-2}$ associated to $c$ and define
$i$ and $j$ by $z_i=E(\sigma,q)$ and
  $z_{j+1}=G(\sigma,q-1)$. Equivalently, $i$ is the smallest rank at which
   we find a bipartition $U_i$ that contains $q-1$ and $q$ in different
   ({\it nonunderlined}) blocks, and $j+1$ is the smallest rank at which we
   find a bipartition $U_{j+1}$ containing $q-1$ and $q$ in the same 
{\it underlined\/}
   block. Thus the chain 
\begin{equation} \label{eq:UKette}
U_0\lessdot U_1\lessdot\cdots\lessdot U_{i-1}<U_{j+1}\lessdot
   U_{j+2} \lessdot\cdots \lessdot U_{3n-2} 
\end{equation}
is $\pi$-compatible {\it and\/}
   $\sigma$-compatible, whereas any $U_k$ of rank $k\in [i,j]$ is only
   $\sigma$-compatible. The intersection $c\cap c'$ 
must be contained in \eqref{eq:UKette},
whence its set of ranks must be disjoint from $[i,j]$. Therefore, the
intersection $c\cap c'$ satisfies condition~(ii) with the above $\pi$.

Conversely, let $\gamma$ be a subchain of $c$ such that the set of its
ranks avoids an interval $[i,j]$, where the interval is one of the
intervals described in item (ii). Then $\gamma$ is a subchain of 
$$U_0\lessdot U_1\lessdot\dots\lessdot
U_{i-1}<U_{j+1}\lessdot U_{j+2}\lessdot\dots\lessdot U_{3n-2}.$$ 
This chain may be extended
   to the $\pi$-compatible maximal chain obtained from $c$ 
by replacing each $U_k$, $i\le k\le j$,
by the bipartitional relation $U_k'$ obtained from $U_k$ by
   swapping the elements $\sigma_{q-1}$ and $\sigma_q$. Therefore $\gamma$
is contained in the intersection of $c$ with  a
   $\pi$-compatible maximal chain, where $\pi$ precedes $\sigma$.
\end{proof}

\begin{theorem}
\label{T_fullc}
The order complex 
\begin{equation} \label{eq:ordercomplex}
\triangle(\Bip(X)\setminus
\{\emptyset,X\times X\}) 
\end{equation}
is homotopy equivalent to a sphere of dimension $\vert X\vert-2$.
\end{theorem}
\begin{proof}
As before, without loss of generality, we may assume that $X=\{1,2,\dots,n\}$.
Lemma~\ref{L_skip} says that the enumeration of maximal faces of
the order complex \eqref{eq:ordercomplex} described in
Steps~1--3 at the beginning of this section grows by creating skipped
intervals. Consequently,
by Theorem~\ref{T_BH1}, in the Morse matching 
constructed by Babson and Hersh~\cite{Babson-Hersh} from such an
enumeration, at
most one critical cell is contributed per maximal chain, namely 
exactly when
the set of intervals $I(c)$ (defined in Definition~\ref{D_ecsi}) 
covers all elements of the set of ranks
$\{1,\ldots, 3n-2\}$.

We are going to show that there is exactly one maximal chain in our
enumeration that contributes a critical cell. This
maximal chain is the lexicographically first chain among the maximal
chains that are compatible with the last permutation in the
Johnson--Trotter enumeration, namely
\begin{equation} \label{eq:lastperm} 
\hat\tau:=(\{2\},\{1\},\{3\},\{4\},\dots,\{n\}).
\end{equation}
According to our construction,
the lexicographically first chain which is compatible with this
permutation is 
\begin{equation} \label{eq:lexfirst} 
\emptyset\prec H_1\prec H_1\vee H_2\prec\dots\prec H_1\vee
H_2\vee\dots\vee H_{3n-2},
\end{equation}
where $H_1,H_2,\dots,H_{3n-2}$ is the enumeration of the
join-irreducible elements in \eqref{E_linext}, with $\sigma$ 
replaced by $\hat\tau$. 
Subsequently, we will compute the dimension of the critical cell
contributed by this maximal chain, using 
the last part of Theorem~\ref{T_BH1}.
The proof will be completed by taking recourse to
Theorem~\ref{thm:Forman}. 

Consider now a $\sigma$-compatible chain $c$, whose associated word 
is $z_1z_2\cdots z_{3n-2}$ (compare Step~2 at the beginning of this
section), and assume
that it contributes a critical cell. Our first goal is to determine
the set of intervals $I(c)$. According to Lemma~\ref{L_skip},
it consists of those intervals listed in items~(i) and (ii) in this lemma
that are minimal with respect
to inclusion. Clearly, all singletons listed in item~(i) of
Lemma~\ref{L_skip} belong to
$I(c)$. Because of the property of $EL$-labellings of distributive
lattices that we recalled in the paragraph before Lemma~\ref{L_skip},
this implies in particular that any other interval $[i,j]$,
$i<j$, can only be minimal with respect to inclusion if the substring
$z_iz_{i+1}\cdots z_{j+1}$ of 
$z_1z_2\cdots z_{3n-2}$ contains no descent. It is then not difficult
to see from the choice of the linear extension \eqref{E_linext} 
of the subposet of join-irreducible
elements (cf.\ Figure~\ref{fig:bipji}) that 
an interval $[i,j]$ listed in item~(ii) of Lemma~\ref{L_skip} 
belongs to $I(c)$ if {\it and only if\/} 
we have
   $z_i\prec z_{i+1}\prec \cdots\prec z_{j+1}$, there is a $q$ such that
$z_i=E(\sigma,q)$ and $z_{j+1}=G(\sigma,q-1)$, 
and the permutation $\pi$ obtained from $\sigma$ by
exchanging the adjacent blocks $\{\sigma_{q-1}\}$ and $\{\sigma_{q}\}$ precedes
   $\sigma$ in the Johnson--Trotter enumeration.

The question remains, how exactly the join-irreducible elements
in \eqref{E_linext} can be aligned in a word
$z_1z_2\cdots z_{3n-2}$ such that the above described minimal
intervals in $I(c)$ cover all of $[1,3n-3]$, and what
properties the permutation $\sigma$ must have.
(While reading the subsequent paragraphs, the reader is advised to
keep Figure~\ref{fig:bipji} in mind.)

In order to answer the above question, we claim that,
if the above described minimal intervals in $I(c)$ cover all of
$\{1,2,\dots,n\}$, 
for all $k=1,2,\dots,n$ we have the following three properties:
\begin{itemize}
\item[(a)] The letters appearing in the union of intervals
\begin{equation}
\label{eq:bigu}
\bigcup_{j=2}^k [E(\sigma,j),G(\sigma,j-1)]
\end{equation}
{\it in the poset of join-irreducible
elements of\/}
$\Bip_\sigma(\{1,2,\dots,n\})$ 
appear in increasing order (with respect to
the linear order described in \eqref{E_linext}) in
$z_1z_2\cdots z_{3n-2}$ (from left to right).
\item[(b)] For
$j=2,\ldots,k$, the permutation obtained from $\sigma$ by exchanging
the blocks $\{\sigma_{j-1}\}$ and $\{\sigma_{j}\}$ precedes $\sigma$
in the Johnson--Trotter enumeration.
\item[(c)] For $j=2,\ldots,k$, the substring $E(\sigma,j)\cdots
   G(\sigma,j-1)$ of $z_1z_2\cdots z_{3n-2}$ contains no descents.
\end{itemize}

We prove these assertions by induction on $k$.
All three statements are vacuously true for $k=1$. 
Assume now that the statements are true for some $k\in\{1,2,\dots,n-1\}$.
We will show that they also hold if we increase $k$ to $k+1$.


Let us for the moment suppose that $k=1$. 
Since $E(\sigma,2)< G(\sigma,1)$  holds in $\Bip(\{1,2,\dots,n\})$, 
the letter $E(\sigma,2)$ must appear before the letter $G(\sigma,1)$
in the word $z_1z_2\cdots z_{3n-2}$, and the
substring $E(\sigma,2)\cdots G(\sigma,1)$  contains an {\it ascent},
i.e., a letter followed by a larger letter. Let $z_l\prec z_{l+1}$ be
the leftmost such ascent.

On the other hand, if $k\ge2$, then, 
by (c), we know that the substring $E(\sigma,k)\cdots\break
G(\sigma,k-1)$ contains no descent. Moreover, since
$E(\sigma,k)<G(\sigma,k)$ in $\Bip(\{1,2,\dots,n\})$, the letter
$G(\sigma,k)$ must appear after the letter $E(\sigma,k)$. From these
facts taken
together, we infer that $G(\sigma,k)$ must appear after
$G(\sigma,k-1)$, and the substring 
$G(\sigma,k-1)\cdots G(\sigma,k)$ contains an ascent.
Let $z_l\prec z_{l+1}$ be
the leftmost such ascent.

To summarize both cases, in $z_1z_2\cdots z_{3n-2}$ we find
\begin{equation} \label{eq:sub1} 
E(\sigma,k+1)\cdots G(\sigma,k-1)\cdots z_lz_{l+1}\cdots G(\sigma,k),
\end{equation}
where $z_l\prec z_{l+1}$ marks the ascent that we identified in both
cases, and where $G(\sigma,k-1)$ is not present if $k=1$. It is allowed
that $z_l=G(\sigma,k-1)$ (or even that $z_l=E(\sigma,2)$ in the case that
$k=1$) or $z_{l+1}=G(\sigma,k)$.

The position $l$ of the ascent $z_l\prec z_{l+1}$ cannot be
covered by a singleton listed under item~(i) in Lemma~\ref{L_skip}, thus
it must be covered by a minimal interval $[i,j]$ listed under item~(ii) in
Lemma~\ref{L_skip}. In particular, $i\le l$.

Assume, by way of contradiction, that the interval $[i,j]$
is not associated to $E(\sigma,k+1)$ and $G(\sigma,
k)$.  Then we must have $z_i=E(\sigma,q)$ and $z_{j+1}=G(\sigma, q-1)$
for some $q\neq k+1$. We claim that $q>k+1$. This is immediate
when $k=1$. For $k\geq 2$, the inequality $q>k+1$ follows from the fact
that criterion (a) does not allow any letter $G(\sigma,q-1)$ satisfying 
$q\leq k$ to appear after $G(\sigma,k-1)$
(recall \eqref{eq:sub1}). 

Let us compare now the
position of the letter $E(\sigma,q)$ with the position of $E(\sigma, 2)$
if $k=1$, or with the position of the letter $G(\sigma,k-1)$ if $k\geq
2$. The letter $E(\sigma,q)$ cannot appear before $E(\sigma,2)$ when
$k=1$, nor can it appear before $G(\sigma,k-1)$ when $k\geq 2$, because
$E(\sigma,q)\succ E(\sigma,2)$, respectively $E(\sigma,q)\succ
G(\sigma,k-1)$, would
force a descent in the substring $E(\sigma,q)\cdots G(\sigma,q-1)$,
making $[i,j]$ nonminimal. However, $E(\sigma,q)$ cannot
appear after $E(\sigma,2)$ when
$k=1$, nor can it appear after $G(\sigma,k-1)$ when $k\geq 2$.
For, if this is the case, then we have an ascent in the substring in 
$E(\sigma,2)\cdots E(\sigma,q)$, respectively 
in the substring $G(\sigma,k-1)\cdots
E(\sigma,q)$, and thus an ascent 
that occurs before $E(\sigma,q)=z_i$, 
a contradiction with our choice that 
$z_l\prec z_{l+1}$ was the leftmost ascent 
in the substring $E(\sigma,2)\cdots
G(\sigma,1)$, respectively in $G(\sigma,k-1)\cdots G(\sigma,k)$
(recall that $i\le l$).

Thus we must have $z_i=E(\sigma,k+1)$ and $z_{j+1}=G(\sigma,
k)$, and the permutation obtained by exchanging the blocks
$\{\sigma_k\}$ and $\{\sigma_{k+1}\}$ must precede $\sigma$ in the
Johnson--Trotter enumeration. This proves that (b) remains valid when we
increase $k$ to $k+1$. Furthermore, the substring
$E(\sigma,k+1)\cdots G(\sigma,k)$ must not contain any descents, proving
that (c) remains valid when we increase $k$ to $k+1$. 
As a consequence, the elements in $[E(\sigma,k+1),G(\sigma,k)]$
in the poset of join-irreducible elements of 
$\Bip_\sigma(\{1,2,\dots,n\})$ 
appear in increasing order in $z_1z_2\dots z_{3n-2}$.
The letter $E(\sigma, k+1)$ must appear before $G(\sigma,k-1)$  since 
$E(\sigma, k+1)<G(\sigma,k-1)$ holds in $\Bip(\{1,2,\dots,n\})$.
This implies that the substrings $E(\sigma,k)\cdots G(\sigma,k-1)$ and
$E(\sigma,k+1)\cdots G(\sigma,k)$ overlap, with $G(\sigma,k)$
following after $G(\sigma,k-1)$. In addition, both substrings
contain no descents. Therefore (a) remains
also valid when we increase $k$ to $k+1$.
This concludes the proof of the properties (a)--(c). 

For $k=n$, statement (b) forces $\sigma$ to be the last permutation in
the Johnson--Trotter enumeration.
Statement (a) implies that $z_1z_2\cdots z_{3n-2}$ must not contain
any descents, making $c$ the first maximal chain in 
$\Bip_\sigma(\{1,2,\dots,n\})$. (Note that, for $k=n$, the union
of intervals \eqref{eq:bigu} contains all $3n-2$ join irreducible
elements of $\Bip_\sigma(\{1,2,\dots,n\})$.) 

Thus far we have shown that the only maximal chain $c$ that {\em may} 
contribute a critical cell is the lexicographically first
$\hat\tau$-compatible chain (that is, the chain \eqref{eq:lexfirst},
with $\hat\tau$ being given in \eqref{eq:lastperm}),
where $\hat\tau$ is the last permutation 
in the Johnson--Trotter enumeration, given in \eqref{eq:lastperm}. 
We show that this chain $c$ contributes a critical cell, by computing
explicitly the
minimal intervals $I(c)$, and then we find the dimension 
of the critical cell by calculating the  simplified system $J(c)$,
as given in Definition~\ref{D_jint}, and by applying
the last part of Theorem~\ref{T_BH1}. 

Since the word $z_1z_2\cdots z_{3n-2}$ 
associated to $c$ contains no descents, it must be the string of
join-irreducible elements listed in the order described in
\eqref{E_linext}, and all intervals $[i,j]$, where $z_i=E(\sigma,k)$ and
$z_{j+1}=G(\sigma,k-1)$, with $k=2,3,\dots,n$, belong to $I(c)$.
Using the list \eqref{E_linext}, we obtain that 
\begin{equation}
\label{eq:ic}
I(c)=
\begin{cases}
\{[1,3]\} & \mbox{if $n=2$;}\\
\{[1,4]\}\cup \bigcup_{k=1}^{n-3} \{[3k,3k+4]\}\cup
\{[3n-6,3n-3]\} & \mbox{if $n\geq 3$.}\\
\end{cases}
\end{equation}

For any $n\geq 2$, the union of intervals contained
in $I(c)$ is $\{1,2,\dots,3n-3\}$, therefore $c$ does contribute a
critical cell by Theorem~\ref{T_BH1}. 
We are left to find the family of intervals $J(c)$, using
Definition~\ref{D_jint} and \eqref{eq:ic} above. It is easy to see
that $J(c)$ is given by 
\begin{equation}
\label{eq:jc}
J(c)=
\begin{cases}
\{[1,3]\} & \mbox{if $n=2$;}\\
\{[1,4]\}\cup \bigcup_{k=1}^{n-3} \{[3k+2,3k+4]\}\cup
\{[3n-4,3n-3]\} & \mbox{if $n\geq 3$.}\\
\end{cases}
\end{equation}

Since $|J(c)|=n-1$, the dimension of the only critical cell is $n-2$ by
the last part of Theorem~\ref{T_BH1}. The theorem now follows upon invoking
Theorem~\ref{thm:Forman}. 
\end{proof}

The above theorem implies in particular that 
$\triangle(\Bip(X)\setminus\{\emptyset,X\times X\}) $ 
has a non-vanishing homology below the top dimension.
This has the following immediate consequence.

\begin{corollary} \label{cor:nshell}
The order complex 
\begin{equation*} 
\triangle(\Bip(X)\setminus
\{\emptyset,X\times X\}) 
\end{equation*}
is not Cohen--Macaulay.
\end{corollary}

If we combine Theorem~\ref{T_fullc} with Philip Hall's 
theorem~\cite[Prop.~3.8.6]{Stanley-EC1} stating that the M\"obius 
function of a graded poset $P$ with minimum element $\widehat{0}$ and
maximum element $\widehat{1}$ is the reduced Euler characteristic of 
the order complex $\triangle(P\setminus
\{\widehat{0},\widehat{1}\})$, then we obtain the following immediate
corollary.

\begin{corollary}
\label{C_MobBip}
The M\"obius function of the minimum and maximal element in $\Bip(X)$ 
is given by
$$
\mu(\emptyset,X\times X)=(-1)^{\vert X\vert}.
$$
\end{corollary}

\section{Regular and irregular intervals in $\Bip(X)$}
\label{sec:int}

In this concluding section, we handle proper intervals of
$\Bip(X)$. 
We distinguish between two kinds of
intervals, regular and irregular ones (see
Definition~\ref{def:regular}.) As we show in
Proposition~\ref{P_regprod}, regular intervals are isomorphic to the
direct product of Boolean lattices and smaller bipartition
lattices. Since the M\"obius function of Boolean lattices is
well-known, and since we computed the M\"obius function of bipartition
lattices in Corollary~\ref{C_MobBip}, it is then easy to compute
the M\"obius function of regular intervals, see
Corollary~\ref{C_Mobprod}. For irregular intervals we prove that their
order complexes are always contractible, see
Theorem~\ref{T_irreg}. Hence, the M\"obius function of an irregular
interval vanishes. We must leave 
the question of the topological structure of
regular intervals open.

\begin{definition} \label{def:regular}
We say that an interval $[U,V]\subseteq \Bip(X)$ 
is {\em regular} if every $x$
satisfying $(x,x)\in V\setminus U$ also satisfies
$\{y\in X\::\: x\inc{U} y\}=\{y\in X\::\: x\inc{V} y\}$. Otherwise we
call $[U,V]$ {\em irregular}.
\end{definition}

In other words, an interval $[U,V]$ is regular if and only
if, for every $x$ belonging to a nonunderlined block in $U$ and to
an underlined block in $V$, the block containing $x$ in $U$ is
equal to the block containing $x$ in $V$ in the ordered bipartition
representations of $U$ and $V$, respectively.
We remark that all of $\Bip(X)$, that is, the interval
$[\emptyset,X\times X]$, is a regular interval. Indeed, assuming
without loss of generality that
$X=\{1,2,\dots,n\}$, we have 
$$[\emptyset,X\times X]=
[U(\{1,2,\dots,n\}),U(\{\underline1,\underline2,\dots,\underline n\})].
$$

\begin{proposition}
\label{P_regprod}
Every regular interval $[U,V]\subseteq \Bip(X)$ 
is isomorphic to a direct product of Boolean
lattices and lattices of the form $\Bip(B)$, where each $B$ is a block
in the ordered bipartition representation of $U$ and of $V$ such that
$B$ is nonunderlined in $U$ and underlined in $B$.
\end{proposition}
\begin{proof}
We prove the statement by induction on the total number of blocks in the
ordered bipartition representations of $U$ and $V$. 

Assume first that there is no $x\in X$ such that 
$x$ is contained in an underlined block of $V$ 
and in a nonunderlined block of $U$. 
By Proposition~\ref{P_cont}, every nonunderlined block of $V$ is
contained in some nonunderlined block of $U$, and every underlined block of
$U$ is contained in some underlined block of $V$. Thus, by our
assumption, $X$ may be
uniquely written as a disjoint union $X=X_1\cup X_2\cup \cdots \cup X_m$
where each $X_i$ is either an underlined block of $V$ (which is also the
union of some consecutive underlined blocks of $U$) or a nonunderlined block
of $U$ (which is also the union of some consecutive nonunderlined blocks
of $V$). Moreover, since every $X_i$ is either a block 
of $U$, or the union of consecutive blocks of $U$, we may order the
blocks $X_i$ in such a way that, for all $i<j$, the relation
$x<_U y$ holds for all 
$x\in X_i$ and $y\in X_j$. Since $U\subseteq V$, it is easy to see 
that, for all $i<j$, $x\in X_i$ and $y\in X_j$, we also have $x<_V y$.
A relation $W\subseteq X\times X$ is bipartitional and belongs to the
interval $[U,V]$ if and only if, for each $i$, the restriction
$\rest{W}{X_i}$ is bipartitional, belongs to
$[\rest{U}{X_i},\rest{V}{X_i}]$, and we have $x<_Wy$
for all $x\in X_i$, $y\in X_j$ satisfying $i<j$. Thus the interval
$[U,V]$ is isomorphic to the direct product $\prod_{i=1}^m
[\rest{U}{X_i},\rest{V}{X_i}]$. 

Given an interval $[\rest{U}{X_i},\rest{V}{X_i}]$, there are two
possibilities: either 
$\rest{U}{X_i}$ has a single nonunderlined block, or
$\rest{V}{X_i}$ has a single underlined block. It is then easily seen
by compressing the (nonunderlined) blocks of $\rest{V}{X_i}$,
respectively the (underlined) blocks of $\rest{U}{X_i}$, to singleton
blocks, that each interval $[\rest{U}{X_i},\rest{V}{X_i}]$ is isomorphic
to a Boolean lattice.

Assume finally that there is at least one $x\in X$ such that $(x,x)\in
V\setminus U$. Since $[U,V]$ is regular, the nonunderlined
block $Y_1$ 
(say) of $U$ containing $x$ is an underlined block of $V$. Let
$Y_0$ ($Y_2$) be the (possibly empty) union of all blocks listed before
(after) $Y_1$ in the ordered bipartition representation of $U$.
Since $Y_1$ is also a block of $V$, and since $U\subseteq V$, it is easy to
see that $Y_0$ ($Y_2$) is also the union of all blocks listed before
(after) $Y_1$ in the ordered bipartition representation of $V$. A relation 
$W\subset X\times X$ is bipartitional and belongs to $[U,V]$ if and only
if for each $i\in\{0,1,2\}$ the restriction $\rest{W}{Y_i}$ is 
bipartitional, belongs to
$[\rest{U}{Y_i},\rest{V}{Y_i}]$ and, for all $i<j$, $x_i\in Y_i$ and
$x_j\in Y_j$ implies $x_i<_Wx_j$. (Here we assume that
the restriction of any relation to the empty set is the empty set, hence 
$Y_i=\emptyset$ implies $|[\rest{U}{Y_i},\rest{V}{Y_i}]|=1$.) 
Thus the interval
$[U,V]$ is isomorphic to $\prod_{i=0}^2 [\rest{U}{Y_i},\rest{V}{Y_i}]$.
We may apply the induction hypothesis to the intervals
$[\rest{U}{Y_0},\rest{V}{Y_0}]$ and $[\rest{U}{Y_2},\rest{V}{Y_2}]$,
whereas the interval $[\rest{U}{Y_1},\rest{V}{Y_1}]$ is isomorphic to
$\Bip(Y_1)$.  
\end{proof}

From the previous proposition, we obtain the following immediate
corollary for the M\"obius function of regular intervals.

\begin{corollary}
\label{C_Mobprod}
Let $[U,V]$ be a regular interval in $\Bip(X)$. Then we have
$$\mu(U,V)=(-1)^{\rank(V)-\rank(U)}.$$
\end{corollary}

\begin{proof}
By \cite[Prop.~3.8.2]{Stanley-EC1}, the M\"obius function behaves
multiplicatively for products of posets. Furthermore, it is well-known
that the M\"obius function of the minimum and maximum element in a
Boolean lattice of rank $m$ is equal to $(-1)^m$. Finally, by
Corollary~\ref{C_MobBip}, we also know that the M\"obius function of the
minimum and maximum element in a bipartition lattice of rank $3n-2$
is equal to $(-1)^n=(-1)^{3n-2}$. If we put all this together and also
recall that the rank function of products of posets is additive, we
obtain the claim.
\end{proof}

Our final theorem says that irregular intervals have a
contractible order complex.

\begin{theorem}
\label{T_irreg}
If $[U,V]\subseteq\Bip(\{1,2,\dots,n\})$ is not regular,
then the order complex
$\triangle([U,V]\setminus \{U,V\})$ is contractible.
In particular, the M\"obius function $\mu(U,V)$ vanishes in 
$\Bip(\{1,2,\dots,n\})$.
\end{theorem}

We show Theorem~\ref{T_irreg} by adapting the proof of
Theorem~\ref{T_fullc}. Again, we need to define a listing of all
maximal chains of $[U,V]$ to which the results of
Section~\ref{sec:Morse} are applicable. As in Section~\ref{s_toc}, the
construction of this listing involves three steps.

\smallskip
{\sc Step 1.}
We list the order complexes $\triangle([U,V]_\sigma\setminus\{U,V\})$,
where $\sigma$ is a permutation of $X$ such that $U$ and $V$ are
$\sigma$-compatible, using the Johnson--Trotter decomposition of
$\triangle([U,V]\setminus\{U,V\})$ as defined in
Definition~\ref{D_Trotterdec}.

\smallskip
{\sc Step 2.}
By Theorem~\ref{T_dist}, for a fixed $\sigma$, the lattice
$\Bip_\sigma(X)$ is distributive, and, hence, also the subposet
$[U,V]_\sigma$ is distributive. As in Step~2 in Section~\ref{s_toc}, 
$[U,V]_\sigma$ has an $EL$-labelling using its join-irreducible
elements, due to \cite[Th.~4.5]{BjoeAA}. 

Evidently, the 
join-irreducible elements of $[U,V]_{\sigma}$ may
be identified with those elements $E(\sigma,i)$, $F(\sigma,i)$, and
$G(\sigma,i)$ which are contained in $V$ but not in $U$. More
precisely, a join-irreducible element $H\in\Bip_\sigma(X)$ is
identified with the join-irreducible element $U\vee H\in[U,V]_\sigma$. 
In the sequel, by abuse of terminology, 
when we speak of ``the join-irreducible elements of $[U,V]_{\sigma}$,"
then we shall always mean the join-irreducible elements
$H\in\Bip_\sigma(X)$ which are contained in $V$ but not in $U$, keeping the
above identification in mind.

For defining the $EL$-labelling, however, we need to start with a
linear extension of the join-irreducible elements of $[U,V]_\sigma$.
Unlike in Section~\ref{s_toc}, we select a different linear extension
for each $\sigma$, the individual choices being independent from each
other. Lemma~\ref{L_nocrit} below describes the details of these
choices. 

Analogously to Section~\ref{s_toc}, we associate to each maximal chain 
$c: U= U_0\lessdot \cdots \lessdot U_{m}=V$
contained in $[U,V]_\sigma$
the word $z_1z_2\cdots z_{m}$ (here $m+1$ is the rank of
$[U,V]_{\sigma}$), where the
letter $z_i$ is the unique join-irreducible element of
$\Bip_\sigma(X)$ contained in $U_i$ but not
contained in $U_{i-1}$.
We list the maximal chains in $[U,V]_\sigma$ according to the
lexicographic order of their associated words.

\smallskip
{\sc Step 3.}
Assume that $\triangle([U,V]_\sigma\setminus\{U,V\})$ and
$\triangle([U,V]_{\sigma'}\setminus\{U,V\})$ appear in the J--T
decomposition of $\triangle([U,V]\setminus\{U,V\})$.
Given a $\sigma$-compatible maximal chain $c$ and a
$\sigma'$-compatible maximal chain $c'$, the chain $c$ precedes $c'$ if
and only if either $\triangle([U,V]_\sigma\setminus\{U,V\})$ precedes 
$\triangle([U,V]_{\sigma'}\setminus\{U,V\})$ in the J--T
decomposition, or if $\sigma=\sigma'$ and $c$ precedes $c'$ in the ordering
of $[U,V]_\sigma$ described in Step~2.

\medskip

\begin{lemma}
\label{L_nocrit}
Let $[U,V]$ be an interval in $\Bip(\{1,2,\dots,n\})$, and let 
$\sigma=(\{\sigma_1\},\{\sigma_2\},\break \dots,\{\sigma_n\})$ be
a permutation such that $\triangle([U,V]_{\sigma}\setminus \{U,V\})$ appears
in the J--T decomposition of $\triangle([U,V]\setminus \{U,V\})$. Assume
that some $p\in \{2,\ldots,n\}$ has the following properties:
\begin{itemize}
\item[(a)] 
Exactly one of $E(\sigma,p)$ and $G(\sigma,p-1)$ is contained
   in $V$ but not in $U$.
\item[(b)] 
At least one of $F(\sigma,p-1)$ and $F(\sigma,p)$ is
   contained in $V$ but not in $U$.
\end{itemize}
Then there is a linear extension of the join-irreducible elements of
$[U,V]_{\sigma}$ such that, no matter what linear extension we select
for the other subcomplexes $\triangle([U,V]_{\pi}\setminus
\{U,V\})$, no maximal
chain of $[U,V]_{\sigma}$ contributes a critical cell.
\end{lemma}
\begin{proof}
Consider first the case where $E(\sigma,p)$ is contained in $V$ but
not in $U$, and $G(\sigma,p-1)$ is either contained in $U$ or
not contained in $V$. Without loss of generality, we may assume that
$F(\sigma,p)$ is contained in $V$ but not in $U$
(otherwise we may simply replace $F(\sigma,p)$ by $F(\sigma,p-1)$ in
the subsequent argument). This means that, when we label the cover
relations $U_1\lessdot V_1$ in $[U,V]_{\sigma}$ by the unique
join-irreducible element of $\Bip_\sigma(\{1,2,\dots,n\})$ that is
contained in $V_1$ but not in $U_1$, the elements
$E(\sigma,p)$ and $F(\sigma,p)$ appear among the labels used but
$G(\sigma,p-1)$ does not. Select a linear extension of the
join-irreducible elements of $[U,V]_\sigma$ (recall the convention
explained in Step~2 after the statement of Theorem~\ref{T_irreg}), 
in which $E(\sigma,p)$ is the least element. This is possible
since $E(\sigma,p)$ is a minimal element among the join-irreducible
elements. We
claim that for this labelling, no maximal chain of $[U,V]_{\sigma}$
contributes a critical cell.  

Consider a maximal chain $c$ in
$[U,V]_{\sigma}$, and assume by way of contradiction that it contributes
a critical cell. By Theorem~\ref{T_BH1}, this means that the set of
intervals $I(c)$ covers all elements of the set of ranks of $[U,V]$.
Let $z_1z_2\cdots z_m$ be the word associated to $c$ according to
Step~2 after Theorem~\ref{T_irreg}. In analogy to Lemma~\ref{L_skip}, a chain contained
in $c$ is also contained in an earlier listed maximal chain if and only
if the set of its ranks is disjoint from at least one of the following
intervals:
\begin{itemize}
\item[(i)] all singletons $[i,i]=\{i\}$ such that $z_i\succ z_{i+1}$,
\item[(ii)] all intervals $[i,j]$ with $z_i=E(\sigma,q)$,
   $z_{j+1}=G(\sigma,q-1)$, for some $q$, such that the permutation 
$\pi$ obtained from $\sigma$ by exchanging
   the adjacent blocks $\{\sigma_{q-1}\}$ and $\{\sigma_{q}\}$ makes
$\triangle([U,V]_{\pi}\setminus \{U,V\})$ precede
   $\triangle([U,V]_{\sigma}\setminus \{U,V\})$ in the J--T
   decomposition of $\triangle([U,V]\setminus \{U,V\})$ as described
in Definition~\ref{D_Trotterdec}.
\end{itemize}
The proof is essentially the same, and is thus omitted.

Next we have to determine the subset $I(c)$ of the above intervals which
are minimal with respect to inclusion. Clearly, $I(c)$ contains all
singletons listed under (i). In the same way as in the proof of
Theorem~\ref{T_fullc}, it can be seen that
an interval $[i,j]$ listed in item~(ii)
belongs to $I(c)$ only if we have
$z_i\prec z_{i+1}\prec \cdots\prec z_{j+1}$. (Note that we do not have
an ``if and only if'' statement anymore, since the ``if'' part in the
proof of Theorem~\ref{T_fullc} followed from the particular choice of
the linear extension \eqref{E_linext} of the poset of join-irreducible
elements, which we do not and need not guarantee in the
present situation.)

Since $E(\sigma,p)<F(\sigma,p)$  holds in $\Bip(\{1,2,\dots,n\})$,
the letter $E(\sigma,p)$ must appear before the letter $F(\sigma,p)$
in $z_1z_2\cdots z_m$. Moreover, because of 
$E(\sigma,p)\prec F(\sigma,p)$, the
substring $E(\sigma,p)\cdots F(\sigma,p)$  contains an ascent. Let
the leftmost such ascent be at position $l$, so that we encounter
$$
E(\sigma,p)\cdots z_lz_{l+1}\cdots F(\sigma,p),
$$
with $z_l\prec z_{l+1}$.
The position $l$ of the ascent is not covered by a
singleton listed under (i) above, thus it must be
covered by a minimal interval $[i,j]$ listed under (ii) above.
In particular, we have $i\le l$.
The interval $[i,j]$ is associated to $z_i=E(\sigma,q)$ and
$z_{j+1}=G(\sigma,q-1)$ for some $q\neq p$ since $G(\sigma,p-1)$ is not
a 
join-irreducible element in $[U,V]_{\sigma}$. Since the interval
$[i,j]$ is minimal, the substring
$E(\sigma,q)\cdots G(\sigma,q-1)$ cannot contain any descents.
By our choice of the linear extension of the join-irreducible
elements of $[U,V]_\sigma$,
we have $E(\sigma,p)\prec E(\sigma,q)$. 
Consider now the relative position of the letters $E(\sigma,p)$ and
$E(\sigma,q)$.
If the letter $E(\sigma,q)$ appears
after $E(\sigma,p)$ in $z_1z_2\cdots z_m$, then we encounter
$$
E(\sigma,p)\cdots E(\sigma,q)\cdots z_lz_{l+1}\cdots F(\sigma,p).
$$
(It is allowed
that $z_l=E(\sigma,q)$ or $z_{l+1}=F(\sigma,p)$.)
Then the substring
$E(\sigma,p)\cdots E(\sigma,q)$ contains an ascent which is not
covered by the
interval $[i,j]$ (recall that $z_i=E(\sigma,q)$), 
in contradiction to having selected the leftmost
ascent in the substring $E(\sigma,p)\cdots F(\sigma,p)$. 
On the other hand, if the letter
$E(\sigma,q)$ appears before $E(\sigma,p)$ in $z_1z_2\cdots z_m$ then
we encounter
$$
E(\sigma,q)\cdots E(\sigma,p)\cdots z_lz_{l+1}\cdots G(\sigma,q-1).
$$
(It is allowed
that $z_l=E(\sigma,p)$ or $z_{l+1}=G(\sigma,q-1)$.)
Because of
$E(\sigma,q)\succ E(\sigma,p)$, there is a descent in the substring
$E(\sigma,q)\cdots E(\sigma,p)$, in contradiction to the fact that the substring
$E(\sigma,q)\cdots G(\sigma,q-1)$ does not contain any descents.

\medskip
Consider now the case where $G(\sigma,p-1)$ is contained in $V$ without
being contained in $U$, and $E(\sigma,p)$ is either contained in $U$ or
not contained in $V$. This case is similar to the previous one, thus we only
outline the necessary changes. Without loss of generality, we may assume that
$F(\sigma,p)$ is contained in $V$, without being contained in $U$.
Take a linear extension of the partial order on the set of
join-irreducible elements of $[U,V]_\sigma$ such
that $G(\sigma,p-1)$ is the maximal element, and consider the word
$z_1z_2\cdots z_m$ associated to a maximal chain
$c$ that contributes a critical cell. Again, by Theorem~\ref{T_BH1}, 
this means that the set of intervals $I(c)$ covers all elements of 
the set of ranks of $[U,V]$. Since
$F(\sigma,p)< G(\sigma,p-1)$ holds in $\Bip(\{1,2,\dots,n\})$, the letter
$F(\sigma,p)$ must appear before the letter $G(\sigma,p-1)$ in
$z_1z_2\cdots z_m$, and the substring $F(\sigma,p)\cdots G(\sigma,p-1)$
contains an ascent. Consider the {\it rightmost\/} such ascent. This ascent must
be covered by an interval $[i,j]$ where $z_i=E(\sigma,q)$ and
$z_{j+1}=G(\sigma,q-1)$ for some $q\neq p$ since $E(\sigma,p)$ is not a
join-irreducible element in $[U,V]_{\sigma}$. Whether the letter $G(\sigma,p-1)$ appears before or after
$G(\sigma, q-1)$ in $z_1z_2\cdots z_m$, we obtain a contradiction, by either
finding an ascent in the substring $G(\sigma,p-1)\cdots G(\sigma, q-1)$
to the right of the supposedly rightmost ascent in
$F(\sigma,p)\cdots G(\sigma,p-1)$, or we find a descent in
$G(\sigma,q-1)\cdots G(\sigma, p-1)$ in contradiction to the
substring $E(\sigma,q)\cdots G(\sigma,q-1)$ not containing any descents.
\end{proof}

\begin{proof}[Proof of Theorem~\ref{T_irreg}]
Assume that $[U,V]$ is irregular. Then there is an
$x\in\{1,2,\break\dots,n\}$ such that $(x,x)\in V\setminus U$ and the block $B$
of $x$ in $U$ is not equal to the block $C$ of $x$ in $V$. 
The goal is to construct an enumeration of all maximal chains of
$[U,V]$ using Steps~1--3 after the statement of the theorem such that
no maximal chain contributes a critical cell according to
Theorem~\ref{T_BH1}. The only undetermined place in Steps~1--3
concerned the choice of linear extension of the join-irreducible elements of
$[U,V]_\sigma$. 

Let $\sigma$
be an arbitrary 
permutation such that $U$ and $V$ are $\sigma$-compatible. 
It suffices to
find a $p\in\{2,\ldots,n\}$ which satisfies the criteria given
in Lemma~\ref{L_nocrit}. List the elements $\sigma_1,\ldots,\sigma_n$,
in this order. The elements of $B$ and $C$ form sublists of consecutive
elements: $B=\{\sigma_{i},\sigma_{i+1},\ldots,\sigma_{j}\}$ and
$C=\{\sigma_{k},\sigma_{k+1},\ldots,\sigma_{l}\}$, for some $i,j,k,l\in
\{1,2,\ldots, n\}$, where $i\leq j$ and $k\leq l$, and where 
the intersection of
the intervals $[i,j]$ and $[k,l]$ is not empty since $x\in B\cap C$.
Thus, since $B\neq C$, one of the following four
situations arises:

\medskip {\sc Case 1.} $i<k$. In this case, we have
$\{\sigma_{k-1},\sigma_k\}\subseteq B$ but
$\{\sigma_{k-1},\sigma_k\}\cap C=\{\sigma_k\}$. Thus,
$G(\sigma,k-1)\not\subseteq V$, while $E(\sigma,k)$ is contained in $V$,
without being contained in $U$. Similarly, $F(\sigma,k)$ is contained in $V$,
without being contained in $U$. We set $p=k$.

\medskip {\sc Case 2.} $i>k$. In this case, we have
$\{\sigma_{i-1},\sigma_i\}\subseteq C$ but
$\{\sigma_{i-1},\sigma_i\}\cap B=\{\sigma_i\}$. Thus,
$E(\sigma,i)\subseteq U$, while $G(\sigma,i-1)$ is contained in $V$,
without being contained in $U$. Similarly, $F(\sigma,i)$ is contained in $V$,
without being contained in $U$. We set $p=i$.

\medskip {\sc Case 3.} $i=k$ and $j<l$. In this case, we have
$\{\sigma_{j},\sigma_{j+1}\}\subseteq C$ but
$\{\sigma_{j},\sigma_{j+1}\}\cap B=\{\sigma_j\}$. Thus,
$E(\sigma,j+1)\subseteq U$, while $G(\sigma,j)$ is contained in $V$,
without being contained in $U$. Similarly, $F(\sigma,j)$ is contained in $V$,
without being contained in $U$. We set $p=j+1$.

\medskip {\sc Case 4.} $i=k$ and $j>l$. In this case, we have
$\{\sigma_{l},\sigma_{l+1}\}\subseteq B$ but
$\{\sigma_{l},\sigma_{l+1}\}\cap C=\{\sigma_l\}$. Thus,
$G(\sigma,l)\not\subseteq V$, while $E(\sigma,l+1)$ is contained in $V$,
without being contained in $U$. Similarly, $F(\sigma,l)$ is contained in $V$,
without being contained in $U$. We set $p=l+1$.

\medskip
We may therefore apply Lemma~\ref{L_nocrit} to conclude that there is
an enumeration of the maximal chains of $[U,V]$ such that, by
Theorem~\ref{T_BH1}, there are no critical cells contributed by the
associated Morse matching. Consequently, by Theorem~\ref{thm:Forman},
the order complex $\triangle([U,V]\setminus \{U,V\})$ is contractible.
\end{proof}

\section*{Acknowledgments}
This project started in 1996 during the authors' stay at the
Mathematical Sciences Research Institute, Berkeley, during the
Combinatorics Program 1996/97, organized by
Louis Billera, Anders Bj\"orner, Curtis Greene, Rodica Simion, and
Richard Stanley. It was completed while the first author was on reassignment
of duties from UNC Charlotte, enjoying the hospitality of the University of
Vienna in Spring 2009. 
We are indebted to an anonymous referee for an extremely careful
reading of the original manuscript and for several valuable
suggestions and references concerning the Morse theory part of this
paper.

\end{document}